\documentclass[11pt]{amsart}
\usepackage[parfill]{parskip}    % Activate to begin paragraphs with an empty line rather than an indent
\usepackage{graphicx, txfonts} %txfonts makes the letters sharp
\usepackage{epstopdf}
\usepackage{amsthm, amssymb, amsfonts, amsmath, enumerate}
\usepackage{mathrsfs} %Needed for \mathscr
\usepackage{relsize} % for resizing in math mode
\usepackage{tikz}

\usepackage{hyperref}

\newcommand{\algt}{\mbox{Alg}_T}

\newcommand{\algtm}{\algt(\M)}

\newcommand{\alg}{\mathsf{Alg}}

\def\Ee{\mathcal{E}}

\def\Set{\mathrm{Set}}

\usepackage{diagrams}
\usepackage[usenames,dvipsnames]{pstricks}
\usepackage{epsfig}
\usepackage{pst-grad} % For gradients
\usepackage{pst-plot} % For axes

\usepackage[all,2cell]{xy}
\input xy
\xyoption{all}
\UseTwocells

\usepackage{diagrams}

\usepackage[usenames,dvipsnames]{pstricks}
\usepackage{pst-grad} % For gradients
\usepackage{pst-plot} % For axes

\newtheorem{theorem}{Theorem}[section]

\newtheorem{proposition}[theorem]{Proposition}

\newtheorem{lemma}[theorem]{Lemma}
\newtheorem{corollary}[theorem]{Corollary}

\theoremstyle{definition}
\newtheorem{definition}[theorem]{Definition}
\newtheorem{defn}[theorem]{Definition}

\newtheorem{example}[theorem]{Example}

\newtheorem{problem}[theorem]{Problem}

\theoremstyle{remark}
\newtheorem{remark}[theorem]{Remark}

\numberwithin{equation}{section}

\newcommand{\bbF}{\mathbb{F}}

\newcommand{\M}{\mathcal{M}}
\newcommand{\calm}{\mathcal{M}}

\newcommand{\calD}{\mathcal{D}}

\newcommand{\Alg}{\mathbb{A}{\rm lg}}
\newcommand{\Cat}{\mathbb{C}{\rm at}}
\newcommand{\fC}{\mathfrak{C}}
\newcommand{\fD}{\mathfrak{D}}
\newcommand{\smallprof}[1]
{\raisebox{.05cm}{\scalebox{0.8}{#1}}}

\newcommand{\uc}{\underline{c}}
\newcommand{\duc}{\smallprof{$\binom{d}{\uc}$}}
\newcommand{\Sigmac}{\Sigma_{\fC}}

\newcommand{\sF}{\mathscr{F}}
\newcommand{\sW}{\mathscr{W}}
\newcommand{\sQ}{\mathscr{Q}}

\renewcommand{\emptyset}{\varnothing}

\renewcommand{\tilde}[1]{\widetilde{#1}}

\newcommand{\algom}{\alg(\O;\M)}

\newcommand{\po}{\ar@{}[dr]|(.7){\Searrow}}
\newcommand{\pb}{\ar@{}[dr]|(.3){\Nwarrow}}
\newcommand{\cat}[1]{\mathcal{#1}}

\DeclareMathOperator{\inj}{inj}

\usepackage[margin=1.5in]{geometry}
\usepackage[colorinlistoftodos]{todonotes}

\newcounter{todocounter}

\def\Ee{\mathcal{E}}

\def\Alg{\mathrm{Alg}}

\def\colim{\mathrm{colim}}

\def\Set{\mathrm{Set}}

\def\Cat{\mathrm{Cat}}

\def\sSet{\mathrm{sSet}}

{\begin{itemize}\tt}%
{\end{itemize}}

\begin{document}

\title{Model structures on operads and algebras from a global perspective}

\author{Michael Batanin}
\address{Mathematical Institute of the Academy \\ \v{Z}itn\'a 25, 115~67 Prague 1, Czech Republic}
\email{bataninmichael@gmail.com}

\author{Florian De Leger}
\address{Mathematical Institute of the Academy \\ \v{Z}itn\'a 25, 115~67 Prague 1, Czech Republic}
\email{de-leger@math.cas.cz}

\author{David White}
\address{Department of Mathematics \\ Denison University
\\ Granville, OH 43023}
\email{david.white@denison.edu}

\thanks{The authors were supported by RVO:67985840 and Praemium Academiae of M. Markl.}

\maketitle

\begin{abstract}
This paper studies the homotopy theory of the Grothendieck construction using model categories and semi-model categories, provides a unifying framework for the homotopy theory of operads and their algebras and modules, and uses this framework to produce model structures, rectification results, and properness results in new settings. In contrast to previous authors, we begin with a global (semi-)model structure on the Grothendieck and induce (semi-)model structures on the base and fibers. In a companion paper, we show how to produce such global model structures in general settings. Applications include numerous flavors of operads encoded by polynomial monads and substitudes (symmetric, non-symmetric, cyclic, modular, higher operads, dioperads, properads, and PROPs), (commutative) monoids and their modules, and twisted modular operads. We also prove a general result for upgrading a semi-model structure to a full model structure.
\end{abstract}

{\small \tableofcontents}

\section{Introduction}

In this paper, we study the homotopy theory of the Grothendieck construction. We begin with a category $\cat B$ and a functor $\Phi: \cat B^{op} \to CAT$ to the large category of categories. Previous authors have induced a model structure on $\int \Phi$ from model structures on the categories $\cat B$ and $\Phi(O)$ (for $O\in \cat B$), under strong hypotheses governing the interaction between the $\Phi(O)$. Our main application of interest is to the case where $\cat B$ is a category of operads and $\Phi(O)$ is the category of $O$-algebras (or left $O$-modules). Unfortunately, in many of these settings, the categories $\Phi(O)$ either do not have model structures or were not previously known to have model structures, prior to the present paper. Even when such model structures exist, often the hypotheses on the categories $\Phi(O)$ are not satisfied.

Consequently, instead of assuming model structures on $\cat B$ and all $\Phi(O)$ exist, then deducing a model structure on $\int \Phi$, we go the opposite way. That is, we assume $\cat B$ is encoded by a polynomial monad (all categories of operads we are interested are), then prove that $\int \Phi$ is encoded by a polynomial monad in sufficient generality to cover our main applications of interest. We define the global model structure on $\int \Phi$, and we prove that, when it exists, we obtain so-called horizontal and vertical model structures on $\cat B$ and $\Phi(O)$. We phrase our work in the maximum possible generality, so that this step applies even when $\int \Phi$ is not encoded by a polynomial monad. We illustrate applications of this general approach, beyond the setting of polynomial monads, in the context of the category of small categories, in Section 3.4.

In \cite[Theorem 4.17]{companion} we determine conditions under which the global model structure exists on $\int \Phi$. However, even without these conditions, there always exists \footnote{Technically, this existence requires our ambient model category $\M$ to be cofibrantly generated, but we will always assume this throughout the paper.} a global semi-model structure, and we prove that its existence implies the existence of horizontal and vertical semi-model structures.

We remind the reader that polynomial monads are equivalent to $\Sigma$-cofibrant colored operads, and most categories of operads can be viewed as algebras of a polynomial monad. In particular, our work provides another way to study the $\Sigma$-cofibrant colored operad whose category of algebras is the category of pairs $(O,A)$ where $O$ is an operad (non-symmetric, symmetric, cyclic, modular, etc.) and $A$ is an $O$-algebra. Furthermore, when this category of pairs has a model structure, our work implies that every category of $O$-algebras does too, even when $O$ is not $\Sigma$-cofibrant. Since the homotopy theory of $\Sigma$-cofibrant operads is much easier, this is a powerful technique. 

The Grothendieck construction has appeared elsewhere in homotopy theory, suggesting many possible applications of our work. When $\cat{B}$ is an indexing category, composing $\Phi$ with the nerve functor $N:Cat \to sSet$, from the category of small categories to the category of simplicial sets, gives rise, via the Grothendieck construction, to a model for the homotopy colimit of the diagram $\Phi$ as the geometric realization of the simplicial replacement of $\Phi$ \cite{bousfield-kan, hollander}. It provides a way to shift from the study of colored operads with a fixed set of colors to the study of colored operads in general 
and an analogous shift in the setting of enriched categories \cite{stanculescu}. The Grothendieck construction has also been used to study slice categories, parameterized spaces and spectra, and global equivariant homotopy theory \cite{harpaz-prasma-integrated}. It also provides a setting for the study of Cisinski's accessible derivators \cite{cis}. More recently, the Grothendieck construction 
has emerged as a setting for the study of spaces of long knots \cite{deleger}, Reedy model categories \cite{cagne-mellies} (which has applications in the theory of Smith ideals \cite{white-yau-5, white-yau6}, and deformation theory \cite{L-infty}. 

Numerous authors have studied the Grothendieck construction using model categories, including Roig \cite{roig}, Stanculescu \cite{stanculescu} (correcting a mistake of Roig), Harpaz and Prasma \cite{harpaz-prasma-integrated}, Balzin \cite{balzin}, and Cagne and Melli\`{e}s \cite{cagne-mellies}. All of these authors begin with model structures on the base and fibers, and then give conditions for the existence of a global model structure on $\int \Phi$ compatible with the given model structures. We avoid these assumptions on the base and fibers, and the model structure we produce on $\int \Phi$ coincides with the model structure produced by these previous papers, when both exist. Often, $\int \Phi$ has a semi-model structure that is provably not a model structure.

We work in the ambient setting of a cofibrantly generated monoidal model category $\M$, and we transfer the (semi-)model structure on $\int \Phi$ from a suitable category of collections built from $\M$. For example, if $\int \Phi$ encodes symmetric operads and their algebras, then we transfer along the forgetful functor $U: \int \Phi \to \M^\Sigma \times \M$ where $\M^\Sigma$ is the category of symmetric sequences in $\M$. If $\int \Phi$ encodes non-symmetric operads and their left modules, we transfer along $U: \int \Phi \to \M^\mathbb{N} \times \M^\mathbb{N}$. Since the (semi-)model structure on $\int \Phi$ arises via the machinery of transfer, it is automatically cofibrantly generated (and combinatorial if $\M$ is). Our approach allows us to prove that $\int \Phi$ is proper, under suitable hypotheses on $\M$.

Furthermore, we investigate how model categorical properties (such as cofibrant generation, properness, and combinatoriality) on $\int \Phi$ imply the same properties on $\cat{B}$ and the categories $\Phi(O)$. 
Sometimes, even though $\int \Phi$ has only a semi-model structure, we are able to produce full model structures on $\cat{B}$ and certain $\Phi(O)$, via a new and powerful technique for promoting a semi-model structure to a full model structure. In particular, this gives a new way to produce model structures on categories of $O$-algebras for various operads $O$.

The generality of our approach allows for applications in the theory of operads, algebras, and modules for a wide variety of operads including non-symmetric, symmetric, cyclic, modular, hyperoperads, and twisted modular operads. We spell out these applications in Section \ref{sec:consequences} and in Section \ref{sec:twisted}. In many cases we recover known results in disparate areas as part of the same general framework, and in some cases we produce model structures where none were known to exist before.

Our approach also works beyond the setting of polynomial monads and substitudes, if one is able to produce a transferred model structure on $\int \Phi$ through other means. We illustrate with an application in the setting of the model category of small categories, where algebras over any finitary monad admit transferred model structures.

We start in Section \ref{sec:preliminaries} with a review of basic definitions and notation that we will use throughout the paper. Most results have appeared elsewhere, but there are new results about semi-model categories in \ref{subsec:semi}. In subsection \ref{subsec:grothendieck-preliminaries}, we prove a general result that allows us to encode the Grothendieck construction $\int \Phi$ as a category of algebras over a polynomial monad, and we provide numerous examples. We also list examples in Theorem \ref{thm:global-model-examples} where we know that the global model structure on $\int \Phi$ exists, including examples from the theory of polynomial monads, substitudes, and commutative monoids (which are not encoded by a polynomial monad) in various ambient model categories. Another example is the case of general finitary monads valued in the 2-category of categories, following \cite{lack}.

In Section \ref{sec:global-model}, we deduce the existence of model structures on $\cat B$ and the $\Phi(O)$ categories from the existence of the global model structure. When $\int \Phi$ is only a semi-model structure (which occurs in great generality), we deduce semi-model structures on $\cat B$ and the $\Phi(O)$. We then study questions of properness, of rectification between the categories $\Phi(O)$, and left cofinal Quillen functors. Applications are sprinkled throughout this section, including to rectification between flexible monads in the setting of \cite{lack}, operadic bimodules, and a general theorem that allows us to provide full model structures on categories $\Phi(O)$ even when $\int \Phi$ is only a semi-model structure (Theorem \ref{thm:semi-to-full}).

In Section \ref{sec:non-tame}, we carefully discuss what can be said when the global semi-model structure exists but is provably not a full model structure (as occurs for symmetric operads). We recover the best known results for symmetric operads via our machinery. The generality of our approach in Sections \ref{sec:global-model} and \ref{sec:non-tame} means that our results will also hold for other flavors of operads, including cyclic operads, $n$-operads, properads, dioperads, PROPs, and other settings detailed in \cite{batanin-berger}. We conclude the paper in Section \ref{sec:twisted} with an application producing a new model structure on the category of twisted modular operads (Theorem \ref{thm:twisted-modular}).

\section{Preliminaries} \label{sec:preliminaries}

We assume the reader is familiar with monoidal model categories \cite{hovey-book} and the basics of colored operads (a.k.a. multicategories) \cite{BM07}. Our main results will have applications in the setting of colored non-symmetric operads, and various flavors of colored symmetric operads. In this section we fix notation and prove some preliminary results.

\subsection{Semi-model structures} \label{subsec:semi}

Suppose $\M$ is a model category, $T$ is a monad on $\M$, and $U$ is the forgetful functor from $T$-algebras to $\M$. A model structure on $T$-algebras is called \textit{transferred} from $\M$ if a morphism $f$ is a weak equivalence (resp. fibration) if and only if $U(f)$ is a weak equivalence (resp. fibration) in $\M$. When $T$ is the free $P$-algebra monad for a colored operad $P$, then we require $\M$ to be a monoidal model category in which $P$ acts. Our set-up will be that $P$ is valued in $\M$, i.e. each object of $P$ is an object of $\M$.

When attempting to transfer a model structure from $\M$ to the category of $T$-algebras, most of the model category axioms are formal. Lemma 2.3 of \cite{SS00} shows how to verify  the retract axiom, the two out of three property, the lifting of a cofibration against a trivial fibration, and the factorization into a cofibration followed by a trivial fibration. For the other lifting and factorization axioms, there are two main approaches. One is to assume $\M$ is extremely nice (e.g. with a coassociative, cocommutative Hopf interval object, and a symmetric monoidal fibrant replacement functor) \cite{BM03}, and then to use the path object argument to do the transfer. Our main applications of interest are not known to satisfy these hypotheses. The other approach is to analyze, for each trivial cofibration $K\to L$ in $\M$, pushouts of $T$-algebras of the form

\begin{align} \label{diagram:pushout}
\xymatrix{
	T(K) \po \ar[r] \ar[d] & T(L) \ar[d] \\
	X \ar[r] & P
}
\end{align}

Such pushouts are called {\em free algebra extensions}. One must show that the pushout morphism $X\to P$ is a weak equivalence. This analysis often involves a complicated filtration \cite{SS00}, \cite{white-commutative-monoids}, \cite{harper-operads}, but does not require strong hypotheses to be assumed on $\M$. Unfortunately, it is often the case that inducting up the filtration only works when $X$ is a cofibrant $T$-algebra (resp. cofibrant in $\M$). In this case, one obtains less than a full model structure on $T$-alg, and ends up with only a semi-model structure (resp. semi-model structure over $\M$). This means that lifting of trivial cofibrations $f$ against fibrations, and factorizations of morphisms $g$ into trivial cofibrations followed by fibrations, only work when $f$ and $g$ have cofibrant domain (resp. domain that forgets to a cofibrant object in $\M$). These notions were introduced in \cite{spitzweck-thesis} under the name $J$-semi model category (resp. $J$-semi-model category over $\M$). Formally, we define:

\begin{defn} \label{defn:semi-model-cat}
A \textit{semi-model structure} on a category $\M$ consists of classes of weak equivalences $\sW$, fibrations $\sF$, and cofibrations $\sQ$ satisfying the following axioms:

\begin{enumerate}
\item[M1] Fibrations are closed under pullback.
\item[M2] The class $\sW$ is closed under the two out of three property.
\item[M3] $\sW,\sF,\sQ$ contain the isomorphisms, are closed under composition, and are closed under retracts.
\item[M4] 
\begin{enumerate}
\item[i] Cofibrations have the left lifting property with respect to trivial fibrations.
\item[ii] Trivial cofibrations whose domain is cofibrant have the left lifting property with respect to fibrations.
\end{enumerate}
\item[M5] 
\begin{enumerate}
\item[i] Every morphism in $\M$ can be functorially factored into a cofibration followed by a trivial fibration. 
\item[ii] Every morphism whose domain is cofibrant can be functorially factored into a trivial cofibration followed by a fibration.
\end{enumerate} 
\end{enumerate}

If, in addition, $\M$ is bicomplete, then we call $\M$ a \textit{semi-model category}. We say $\M$ is \textit{cofibrantly generated} if there are sets of morphisms $I$ and $J$ in $\M$ such that $\inj I$ is the class of trivial fibrations, $\inj J$ is the class of fibrations in $\M$, the domains of $I$ are small relative to $I$-cell, and the domains of $J$ are small relative to morphisms in $J$-cell whose domain is cofibrant. We will say $\M$ is \textit{combinatorial} if it is cofibrantly generated and locally presentable.
\end{defn}

We note that, in a cofibrantly-generated semi-model category, M1 is redundant. In any semi-model category, trivial fibrations are closed under pullback because they are characterized as morphisms having the right lifting property with respect to cofibrations.

Often, semi-model structures are produced by transfer from a model structure along an adjunction. In this case, sometimes a lesser cofibrancy assumption is required in M4(ii) and M5(ii).

\begin{definition} \label{defn:semi-model-over}
Assume there is an adjunction $F:\calm \rightleftarrows \calD:U$ where $\calm$ is a cofibrantly generated model category, $\calD$ is bicomplete, and $U$ preserves colimits over non-empty ordinals. We say an object $A$ in $\calD$ is \textit{cofibrant in $\calm$} if $U(A)$ is cofibrant in $\calm$. We say that $\cat{D}$ is a \textbf{semi-model category over $\calm$} if $\calD$ satisfies a modification of the axioms of Definition \ref{defn:semi-model-cat}, replacing M4(ii) and M5(ii) with the first two axioms below:

\begin{itemize}
\item M4(ii)': Trivial cofibrations whose domain is cofibrant in $\M$ have the left lifting property with respect to fibrations.
\item M5(ii)': Every morphism whose domain is cofibrant in $\M$ can be functorially factored into a trivial cofibration followed by a fibration.
\end{itemize}
\end{definition}

Some authors also require the following condition:
\begin{itemize}
\item M6: The initial object of $\calD$ is cofibrant in $\M$, and any cofibration $f:A\to B$ in $\calD$ whose domain is cofibrant in $\M$ is mapped to a cofibration $U(f)$ in $\M$.
\end{itemize}

In the context of operads, this axiom is related to the question of when a cofibrant algebra forgets to an underlying cofibrant object. Some semi-model categories of operad-algebras satisfy this axiom while others don't, e.g., the category of commutative monoids in a monoidal model category that satisfies the commutative monoid axiom but not the strong commutative monoid axiom \cite{white-commutative-monoids}. Hence, we will not require M6, because we want our work to apply in maximal generality. Note that the weaker definition in \cite{fresse-book} (12.1.1) is what Spitzweck would call an $(I,J)$-semi model category \cite{spitzweck-thesis}. We have no need of this notion, since our applications will be $J$-semi model categories, but of course results proven with the weaker definition will be true for our setting. Our main source of examples of semi-model structures will be Theorems 2.2.1 and 8.1.7 of \cite{Reedy-paper}, Theorem 6.2.3 and 6.3.1 of \cite{white-yau}, and Corollary 3.8 of \cite{white-commutative-monoids}.

Sometimes we can get from a semi-model structure to a full model structure using the following lemma (which has not appeared in the literature before).

\begin{lemma} \label{lemma:factorization-suffices}
Suppose $\M$ is a cofibrantly generated semi-model category. Suppose any morphism admits a factorization into a trivial cofibration followed by a fibration. Then $\M$ is a model category, i.e., any trivial cofibration lifts against any fibration.
\end{lemma}

\begin{proof}
Let $J$ denote the set of generating trivial cofibrations. Assuming the factorization posited in the hypotheses, we prove that $J$-cell is contained in the weak equivalences, and furthermore that all morphisms in $J$-cell have the left lifting property with respect to the fibrations. A related result is Theorem 2.2 in \cite{spitzweck-thesis}, but we give a self-contained account here without requiring any results from \cite{spitzweck-thesis}.
	
Given $a\in J$, factor $a$ as $p\circ j$ where $j$ is a trivial cofibration and $p$ is a fibration. Since $\M$ is cofibrantly generated, $p$ is in $J$-inj. This implies $a$ has the right lifting property with respect to $p$. So $a$ is a retract of $j$, hence is a weak equivalence.

Now suppose $b:C\to Y:=C\coprod_A B$ is a pushout of $a\in J$. Since the fibrations have the right lifting property with respect to $J$, this implies the trivial fibrations do too, so $a$ is a cofibration \cite[Lemma 1.7.1]{barwickSemi}. Hence, $b$ is a cofibration. Factor $b$ into a trivial cofibration $i: C\to X$ followed by a fibration $p:X\to Y$. Consider the following diagram:
	
	\begin{align*}
	\xymatrix{A\ar[r] \ar[d]_a \po & C \ar[d] \ar[r] & X\ar[d]^p \\ B\ar[r] \ar@{..>}[urr] & Y \ar@{=}[r] & Y}
	\end{align*}

As $p$ is in $J$-inj, the dotted lift exists. Thus, both $B$ and $C$ map to $X$ and the diagram formed by the upper triangle commutes, so this induces a morphism from the pushout $Y$ to $X$ filling the second square. Now apply the retract argument \cite[Lemma 1.1.9]{hovey-book} to see that $b$ is a retract of the trivial cofibration $i$ and hence is a weak equivalence. We see that pushouts of morphisms in $J$ are weak equivalences (hence trivial cofibrations). In fact, the same argument constructs a lift for $b$ against any fibration. 

Now suppose $X:\lambda \to \M$ is a $\lambda$-sequence in which each morphism is a pushout of a morphism in $J$. We have seen that each morphism is thus a cofibration, a weak equivalence, and has the left lifting property with respect to fibrations. The transfinite composition is also a cofibration, because cofibrations are characterized by the left lifting property with respect to trivial fibrations \cite[Lemma 1.7.1]{barwickSemi}. We will show that the composite $f_{\lambda}:X_0\to X_\lambda$ has the left lifting property with respect to any fibration $p:E\to Z$. The composition of any finite number of morphisms with the left lifting property with respect to fibrations is again such a morphism, by induction. For a limit ordinal, suppose $g_{\beta}:X_\beta \to E$ are the existing lifts, for all $\beta < \lambda$. As $X_\lambda = \colim_{\beta < \lambda} X_\beta$, the $g_\beta$ induce a morphism $g_\lambda: X_\lambda \to E$. This $g_\lambda$ is a lift: 
\[
\xymatrix{
X_0 \ar[r] \ar[d]_{f_\lambda} & E\ar[d]^p \\
X_\lambda \ar[r] \ar@{..>}[ur]^{g_\lambda} & B}
\]
making the two triangles above commute, by the universal property of the colimit. So we see that all morphisms in $J$-cell have the left lifting property with respect to fibrations.

We would like to know that $X\to X_\lambda$ is a weak equivalence. We may factor it into a trivial cofibration followed by a fibration, and then the retract argument tells us it is a weak equivalence. 
	
Finally, we must prove that any trivial cofibration has the left lifting property with respect to fibrations. The small object argument factors any trivial cofibration $a$ into $p\circ j$, a morphism $j$ in $J$-cell followed by a fibration. We have seen that the morphism $j$ must be a weak equivalence, so now the two out of three property shows us that $p$ is a trivial fibration. That $a$ is a cofibration allows us to lift $a$ with respect to $p$, and so the retract argument \cite[Lemma 1.1.9]{hovey-book} shows us that $a$ is a retract of $j$. Thus, $a$ has the left lifting property with respect to any morphism that $j$ does. 
In particular, $a$ will lift against any fibration.
\end{proof}

In fact, it suffices to check even less, as the following lemma demonstrates.

\begin{lemma} \label{lemma:factorization-implies-full}
Let $\M$ be a semi-model category. If, for every cofibration $g:X\to Z$, there is a factorization into a trivial cofibration followed by a fibration, then $\M$ is a model category.
\end{lemma}

\begin{proof}
By Lemma \ref{lemma:factorization-suffices}, it is sufficient to verify that any morphism $f:X\to Y$ has a factorization into a trivial cofibration followed by a fibration. Factor $f$ into a cofibration $g:X\to Z$ followed by a trivial fibration. Factor $g$ into a trivial cofibration $X\to W$ followed by a fibration $W\to Z$, using the hypothesis. Observe that the composite $W\to Z\to Y$ is still a fibration, so we are done.
\end{proof}

\begin{remark}
If $\M$ is a semi-model category cofibrantly generated by sets $I$ and $J$, where the domains of morphisms in $J$ are cofibrant, then the first factorization in the proof of Lemma \ref{lemma:factorization-suffices} exists, the second exists if $C$ is cofibrant, and the third exists if $X_0$ is cofibrant. This explains why pushouts in $J$-cell into cofibrant objects, and transfinite compositions between cofibrant objects, are well-behaved. This fact is critical in the proof of the main theorem of \cite{bous-loc-semi}, and examples of semi-model categories of this sort (with domains of $J$ cofibrant) are listed in \cite{crm,bous-loc-semi,white-oberwolfach,Reedy-paper}.
\end{remark}

Powerful techniques are developed in \cite{batanin-berger} and \cite[Section 4]{companion} for transferring model structures to categories such as $\int \Phi$, and also proving that these model structures are left proper. Sometimes, these model structures satisfy a weaker version of left properness that we now recall from \cite{batanin-berger}.

\begin{defn} \label{defn:relatively-left-proper}
A model structure on $\algtm$ is \textit{relatively left proper} if, for every weak equivalence $f:R\to S$, and any cofibration $R\to R'$, if $R$ and $S$ forget to cofibrant objects in $\M$ then the pushout $R' \to R' \cup_R S$ is a weak equivalence of $T$-algebras.
\end{defn}

\subsection{Polynomial monads}

In this section we recall the theory of polynomial monads \cite{batanin-berger}. Polynomial functors have been studied extensively in category theory (e.g. in \cite{kock-poly}). A polynomial monad $T$ on a comma category $\Set/I$ is a monad in the 2-category of overcategories, polynomial functors, and cartesian natural transformations \cite[Definition 6.3]{batanin-berger}. In particular, this implies that $T$ has a decomposition as $ t_! \circ p_* \circ s^*: \Set/I \to \Set/I$ generated by a polynomial:

\[
\xymatrix{
	I & E \ar[l]_s \ar[r]^p & B \ar[r]^t & I
}
\]

We will always assume the map $p$ has finite fibers. Recall that the category of polynomial monads on $\Set/I$ is equivalent to the category of $\Sigma$-free $I$-colored operads \cite{batanin-kock-weber}.

\begin{remark}
When $T$ is a polynomial monad, the category of $T$-algebras always inherits a transferred semi-model structure, by Theorem 6.3.1 of \cite{white-yau}. Sometimes, $Alg_T$ has the additional structure of a semi-model structure over $\Ee$, e.g., when $Alg_T$ is the category of symmetric operads. Thus, in Section \ref{sec:global-model} we will investigate situations featuring $\int \Phi$ where this additional structure is present.
\end{remark}

Given any polynomial monad $T$, we can create a filtration to analyze the critical pushout required for transferring a model structure to $\Alg_T$. This is spelled out in \cite[Proposition 4.6]{companion}. In \cite{batanin-berger}, such a filtration was used for particularly nice polynomial monads, there called \textit{tame polynomial monads} \cite[Definition 6.19]{batanin-berger}, to produce full model structures on $\Alg_T$ that are furthermore left proper (resp. relatively left proper) for nice-enough ambient model categories $\M$ \cite[Theorem 8.1]{batanin-berger}. This theory was generalized in \cite[Section 4]{companion} to define \textit{quasi-tame polynomial monads}, and prove that their categories of algebras have transferred full (resp. left proper, resp. relatively left proper) model structures. Indeed, examples in \cite[Remark 4.19, Example 4.28]{companion} show that without quasi-tameness, only semi-model structures are possible. The current paper can be read without knowing the definitions of tame or quasi-tame, but we sometimes use this terminology when referencing results from \cite{companion}.

\subsection{The Grothendieck construction} \label{subsec:grothendieck-preliminaries}

In this section, we record some basic observations about the Grothendieck construction that we will need.

Let $\cat{B}$ be a category and 
$$\Phi:\cat{B}^{op} \to CAT$$
be a functor.
Let $\phi:O\to O'$ be a morphism in $\cat{B}.$ We will denote  $\phi^* =\Phi(\phi): \Phi(O')\to \Phi(O).$ We can form now the Grothendieck construction $\int \Phi$: the objects are pairs $(O,A)$ where $O\in  \cat{B}$ and $A\in \Phi(O).$ 
A morphism 
$$(O,A)\to (O',A')$$
is a pair $(\phi,f)$ where $\phi:O\to O'$ and $f:A\to \phi^*(A').$ 
We will very often identify a morphism $f$ in $\Phi(O)$ with a morphism $(id,f)$ in $\int \Phi.$ 

As we are interested in (semi-)model structures, we will assume $\int \Phi$ is cocomplete. This implies that each $\phi^*$ is a right adjoint. The Grothendieck construction always has a projection functor $p:\int \Phi \to \cat{B}$ taking $(O,A) \to O$. If each category $\Phi(O)$ has an initial object $i_O$, then this functor $p$ admits a left adjoint $i:\cat{B} \to \int \Phi$ by $O \mapsto (O,i_O)$. In our settings these will form a Quillen pair \cite[Definition 1.12]{barwickSemi}, since $p$ preserves (trivial) fibrations by definition of global and horizontal weak equivalences and fibrations. If each category $\Phi(O)$ has a terminal object $t_O$, then this functor $p$ admits a right adjoint $r:\cat{B} \to \int \Phi$ by $O \mapsto (O,t_O)$. In our settings, this functor $r$ is a right Quillen functor, by the definition of global and horizontal structures.

Sometimes, the functor $i$ admits a further left adjoint $E:\int \Phi \to \cat{B}$. This is the case for certain categories of operads and their algebras, and $E(O,A) = O_A$ is the \textit{enveloping operad}, whose algebras are $O$-algebras under $A$ \cite{BM07}. This is also the case for categories of (commutative) monoids and modules, and $E(R,M) = R_M$ is the \textit{enveloping algebra} of the pair $(R,M)$, where $M$ is an $R$-module. Note also that the functor $E$ does not exist for all categories of operads. For example, if we work with reduced operads (meaning $P(0) = \ast$ is the terminal object), then there is no enveloping operad construction, since $O_A$ is not a reduced operad ($O_A(0) \cong A$).

These functors and the adjunctions are summarized in Figure (\ref{figure:adjoints}):

\begin{align} \label{figure:adjoints}
\xymatrix{
	\int \Phi
	\ar@<0.75cm>@{}[d]|{\dashv} % this an empty arrow, just the adjunction symbol. The "|" means put the symbol in the middle of the arrow (not above or below)
	\ar@/^0.5cm/[d]_{p} % this is how to do a curved arrow going down. The 0.5cm is unit of length, and curvature automatically matches it. Probably it's distance to the peak.
	\ar@{}[d]|{\dashv} 
	\ar@<-0.75cm>@{}[d]|{\dashv} % "negative" sign flipped it to the left; leftmost adj.
	\ar@/_1cm/[d]_{E} % this is a sharply curved arrow. The "_" means concave up (from the arrow's pov)
	\\
	\cat{B}
	\ar@/_1cm/[u]_{r} % this is a sharply curved arrow. The "_" means concave up
	\ar@/^0.5cm/[u]_{i} % now needs to be concave down
}
\end{align}

\subsection{Polynomial monad for the Grothendieck construction}

In this section, we summarize a result from \cite[Proposition 3.2]{companion}, where we prove that, if $\cat{B}$ is the category of algebras over a polynomial monad $T$, and if there is a reasonable way to speak of algebras over every $O\in \cat B$, then $\int \Phi$ is encoded as algebras over a polynomial monad.

Let $T$ be an $I$-colored symmetric operad in a closed symmetric monoidal category $\M$, equipped with a morphism of operads $\phi: T \to SOp(J)$, where $SOp(J)$ is the $\Sigma$-free symmetric operad for $J$-colored symmetric operads. This morphism $\phi$ induces a restriction functor
$$\phi^*: \Alg_{SOp(J)}(\M) \to \Alg_T(\M).$$
This restriction functor allows us to talk about algebras of $O$, when $O$ is itself an algebra of $T$ \cite[Section 3.1]{companion}. The result is a functor 
$$\Phi:\Alg_{T}^{op}\to \Cat$$
which assigns to a $T$-algebra $O$ the category of $O$-algebras and to a morphism of $T$-algebras $f:O\to O'$ assigns the restriction functor $f^*: \Alg_{O'}(\M)\to \Alg_O(\M).$ We can form now the Grothendieck construction $\int \Phi.$ So, the objects of $\int \Phi$ are pairs $(O,A)$ where $O\in \Alg_T(\M)$ and $A\in \Alg_O(\M).$ A morphism $$(O,A)\to (O',A')$$ is a pair $(f,\alpha)$ where $f:O\to O'$ and $\alpha:A\to f^*(A').$ 

Let now $\M$ be a symmetric monoidal model category. We will call a morphism $(f,\phi)$ in $\int \Phi $ a weak equivalence (fibration) if the underlying morphisms of $J$-collections $U(f)$ and $I$-collections $U(\phi)$ are pointwise weak equivalences (fibrations).
We will say that $\int \Phi$ admits the {\em global model structure} if there is a model structure on $\int \Phi$ with weak equivalences and fibrations defined as above \cite{harpaz-prasma-integrated}. The following is proven in \cite[Proposition 3.2]{companion}.

\begin{proposition} \label{prop:poly-for-Gr}
If $\phi$ is a morphism of polynomial monads in $Set$ then there exists a polynomial monad $Gr(T)$ such that the category $\int \Phi $ is isomorphic to the category $\Alg_{Gr(T)}(\M).$ 
\end{proposition}

Let $ORTr(J)$ denote the set of ordered rooted trees whose edges are decorated by $J$. For completeness, we point out that, if $T$ is represented by a polynomial $I \leftarrow E \to B \to I$ and $\phi$ is given by a morphism of polynomials which includes a morphism $\psi: B \to ORTr(J)$, then the polynomial monad for $Gr(T)$ is:
$$\xymatrix{I\sqcup J &\ar[l]_{\xi}D^*\ar[r]^{\pi}&D\ar[r]^{\tau}&I\sqcup J}$$

where $D$ is the coproduct $$B \sqcup B = B \sqcup \{(b,\sigma)\ | \ \sigma \in ORTr(J), b\in \psi^{-1}(\sigma)\},$$ 
the target of an element $b\in B$ is $t(b)$, the target of $(b,\sigma)$ is the target of $\sigma$, and the set $D^*$ is the set of elements of $D$ with a marked source in the usual way. For full details, we refer the reader to \cite[Section 3.1]{companion}.	

\subsection{Global admissibility}

The following theorem is proven in \cite{companion} as Theorems 5.1 and 6.1. We state it here so that we can deduce consequences from it using the results below. For model-categorical terminology, we refer the reader to \cite{batanin-berger, companion}.

\begin{theorem} \label{thm:global-model-examples}
Assume $\M$ is a compactly generated monoidal model category satisfying the monoid axiom. Then the following global model structures exist:
\begin{enumerate}
\item The category of pairs $(R,M)$ where $R$ is a monoid (or, more generally, an $A$-algebra for a commutative monoid $A$) and $M$ is an $R$-module (left, right, or bimodule).
\item The category of pairs $(O,A)$ where $O$ is a non-symmetric operad and $A$ is an $O$-algebra.
\item The category of pairs $(O,M)$ where $O$ is a non-symmetric operad and $M$ is a left $O$-module.
\item The category of pairs $((O,P),M)$ where $O$ and $P$ are non-symmetric operads, and $M$ is an $O-P$-bimodule.
\item The category of pairs $(O,M)$ where $O$ is a constant-free symmetric operad (meaning $P(0)=\emptyset$ is the initial object of $\M$) and $M$ is a constant-free left $O$-module.
\item The category of pairs $(O,M)$ where $O$ is a constant-free $n$-operad and $M$ is a constant-free left $O$-module.
\item The category of pairs $(R,M)$ where $R$ is a commutative monoid and $M$ is an $R$-module, if we assume $\M$ additionally satisfies the commutative monoid axiom \cite{white-commutative-monoids}.
\end{enumerate}

If $\M$ is $h$-monoidal then these model structures are relatively left proper. If $\M$ is strongly $h$-monoidal then they are left proper.
\end{theorem}

We note that example (5) above requires the theory of substitudes and is related to the Baez-Dolan stabilization hypothesis \cite{Reedy-paper}.

The conditions of this theorem are satisfied by a wide range of model categories, documented in \cite{batanin-berger}, \cite{white-localization}, \cite{white-yau}, \cite{white-yau2}, including simplicial sets, spaces, spectra, equivariant and motivic spectra, chain complexes, the stable module category, the category of small categories, simplicial presheaves, and various models for categories of weak $n$-categories.

There are many more examples where $\int \Phi$ admits a full global model structure, and we cover a few in Section \ref{sec:global-model}. Furthermore, in \cite[Section 6]{companion} we obtain left properness and relative left properness results for the model categories of commutative monoids and for $\int \Phi$ in this setting, using the machinery of Section \ref{sec:global-model}.

When $\cat{B}$ is the category of non-symmetric operads, there cannot be a transferred model structure on $\int \Phi$ in general, because, as we will see, this would imply the existence of a transferred model structure on $\Phi(O)$ for every $O$, including $O = Com$ for which counterexamples are well-known \cite{white-commutative-monoids}. Nevertheless, we have a semi-model structure by \cite[Theorem 6.3.1]{white-yau} and we will investigate the consequences of this in the next section.

\section{The global (semi-)model structure} \label{sec:global-model}

In this section, we will prove that, if $\int \Phi$ admits a transferred model structure, then the same is true for the base $\cat B$ and the fibers $\Phi(O)$. In fact, we will prove much more. Motivated by applications to the homotopy theory of operads, we prove that properness of $\int \Phi$ implies properness of the model categories $\cat B$ and $\Phi(O)$, we deduce a rectification result, and we prove analogous results for when $\int \Phi$ is only a semi-model structure. We will see in Section \ref{sec:non-tame} that the semi-model categorical case is unavoidable.

Let $\cat{B}$ be a category and 
$$\Phi:\cat{B}^{op} \to CAT$$
be a functor.

\begin{definition} \label{defn:homotopically-structured}.
We say $\Phi$ is {\it homotopically structured} if $\cat{B}$ is equipped with two classes of morphisms called {\it horizontal} weak equivalences and fibrations, and for each $O\in \cat{B}$ the category $\Phi(O)$ is also equipped with classes of weak equivalences and fibrations called {\it vertical}.  We will call a morphism $(\phi,f)$ in $\int\Phi$ a {\it global weak equivalence (fibration)} if $\phi$ is a horizontal and $f$ is a vertical weak equivalence (fibration). We also define a class of {\it global cofibrations} by the left lifting property with respect to trivial global fibrations.
\end{definition} 

A {\it global trivial cofibration} is a global cofibration that is also a global weak equivalence. Note that we do not know in general that global trivial cofibrations are characterized by the left lifting property with respect to global fibrations.

We will say that $\int \Phi$ admits a {\it global} (semi-)model structure if there is a (semi-)model structure on $\int\Phi$ with weak equivalences, fibrations, and cofibrations defined as above. It also may happen that there is a model structure on a given category $\Phi(O)$ where weak equivalences and fibrations are vertical, and vertical cofibrations are defined by the left lifting property.  We call such a model structure {\it vertical}. One can also have a {\it horizontal} model structure on $\cat{B}$ using horizontal weak equivalences, fibrations, and cofibrations. 

\subsection{From global to horizontal and vertical}

In \cite{companion}, we produce global (semi-) model structures in very general contexts (see Theorem \ref{thm:global-model-examples}). We will use the following results to deduce the existence of horizontal and vertical (semi-)model structures. We begin with a characterization of the global (trivial) cofibrations.

\begin{lemma} \label{lemma:cof-of-pairs}
Suppose $\Phi$ is homotopically structured and $(\phi,f):(O,A)\to (O',A')$ is a global (trivial) cofibration. Then $\phi$ is a horizontal (trivial) cofibration. Suppose $(id,f):(O,A)\to (O,A')$ is a (trivial) cofibration of pairs. Then $f:A\to A'$ is a vertical (trivial) cofibration.
\end{lemma}

\begin{proof}
Let $(\phi,f):(O,A)\to (O',A')$ be a global cofibration.
Suppose $\beta:O_1\to O_2$ is a horizontal trivial fibration, and we are given a lifting problem:
\[
\xymatrix{
O\ar[r] \ar[d]_\phi & O_1\ar[d]^\beta \\
O' \ar[r] & O_2
}
\]
Consider the global trivial fibration $(\beta,t):(O_1,t_1)\to (O_2,t_2)$ where $t_j$ refers to the terminal object in $\Phi(O_j)$. Consider the lifting problem: 
\[
\xymatrix{
(O,A) \ar[r] \ar[d]_{(\phi,f)} & (O_1,t_1) \ar[d]^{(\beta,t)} \\
(O',A')\ar[r] & (O_2,t_2)
}
\]
Then the first component of the lift $(O',A') \to (O_1,t_1)$ demonstrates that $O\to O'$ has the right lifting property in the horizontal structure, and hence is a horizontal cofibration. If $(\phi,f)$ is also a weak equivalence, then so is $\phi$, by Definition \ref{defn:homotopically-structured}.

Next, let $(id_O,f)$ be a global cofibration and suppose $g:X\to Y$ is a vertical trivial fibration in $\Phi(O)$. Then $(id_O,g):(O,X)\to (O,Y)$ is a global trivial fibration. The lift of $(\phi,f)$ against $(id_O,g)$ gives a lift of $f$ in $\Phi(O)$, proving that $f$ is a vertical cofibration. If $(id_O,f)$ is also a weak equivalence, then so is $f$, by Definition \ref{defn:homotopically-structured}.

\end{proof}

The following characterization of the global (trivial) cofibrations is related to a result (with different conditions) in \cite{stanculescu}. This characterization of global cofibrations is also used in \cite{cagne-mellies}, in the situation where $\cat B$ and every fiber $\Phi(O)$ have model structures, under the name `total cofibrations.'

\begin{proposition} \label{prop:stanculescu}
Suppose $\Phi$ is homotopically structured (resp. admits the global model structure) and assume, for every $\gamma$ in $\cat{B}$, that $\gamma^*$ preserves fibrations and trivial fibrations. 
A morphism $(\phi,f): (O,A)\to (P,B)$ is a global cofibration (resp. trivial cofibration) if and only if $\phi$ is a horizontal (trivial) cofibration and ${}^\dashv f:\phi_!(A) \to B$ is a vertical (trivial) cofibration.
\end{proposition}

\begin{proof}
Suppose $\phi: {O} \to {P}$ is a horizontal cofibration and $\phi_!(A) \to B$ is a vertical cofibration in ${P}$. In order to show $(\phi,f)$ is a global cofibration, we proceed via lifting. Consider the following lifting diagram, where the labels $\alpha$ and $\xi$ denote horizontal morphisms only:
\begin{align} \label{diagram:stan}
\xymatrix{
({O},A) \ar[r]^{(\beta,b)} \ar[d]_{(\phi,f)} & ({Q},C) \ar@{->>}[d]^\alpha \\
({P},B) \ar[r]_{\xi} \ar@{..>}[ur]^\gamma & ({Z},D)
}
\end{align}

Because horizontal trivial fibrations are inherited from global ones, there is a horizontal lift $\gamma:{P} \to {Q}$. Next, we obtain a vertical lift by setting up the following lifting problem in $\Phi({P})$, using that $\xi^*(D) \cong \gamma^*(\alpha^*(D))$:
\begin{align} \label{diagram:stan2}
\xymatrix{
\phi_!(A) \ar[r] \ar[d] & \gamma^*(C) \ar[d]\\
B \ar[r] & \gamma^*(\alpha^*(D))
}
\end{align}

Since $\gamma^*$ preserves trivial fibrations, the right vertical morphism is a trivial fibration, and the left vertical morphism is a cofibration by hypothesis. The lift $g:B\to \gamma^*(C)$ completes the global lift in (\ref{diagram:stan}), and proves $(\phi,f)$ is a global cofibration. Commutativity of the lower triangle in (\ref{diagram:stan}) is commutativity of the lower triangle in (\ref{diagram:stan2}). Commutativity of the upper triangle in (\ref{diagram:stan}) requires that $b=g\circ f: A \to \phi^*(B) \to \phi^*\gamma^*(C)$. This is true because the adjoint diagram (upper triangle of (\ref{diagram:stan2})) commutes.

Conversely, suppose $(\phi,f)$ is a global cofibration. By Lemma \ref{lemma:cof-of-pairs}, $\phi$ is a horizontal cofibration. To prove $\phi_!(A)\to B$ is a vertical cofibration in $\Phi({P})$, set up a lifting problem
\begin{align} \label{diagram:stan3}
\xymatrix{
\phi_!(A) \ar[r] \ar[d] & X \ar@{->>}[d]\\
B \ar[r] & Y
}
\end{align}

Convert this into a global lifting problem and introduce $({O},A)$ via the unit morphism $\epsilon:A\to \phi^* \phi_!(A)$

\begin{align} \label{diagram:stan4}
\xymatrix{
({O},A) \ar[r]^{(\phi,\epsilon)} \ar@{^(->}[dr]_{(\phi,f)} & ({P},\phi_!(A)) \ar[r] \ar[d] & ({P},X) \ar@{->>}[d]\\
& ({P},B) \ar[r] & ({P},Y)
}
\end{align}

Since $(\phi,f)$ is a cofibration, there is a lift $({P},B) \to ({P},X)$ making the lower triangle and left trapezoid commute. It is easy to see that this lift also makes the middle triangle commute, since the following triangle commutes in $\Phi(P)$:

\[
\xymatrix{
\phi_!(A) \ar[r] \ar[d]^{{}^\dashv f} & X \\
B \ar[ur] & \\
}
\]
where ${}^\dashv f$ is the mate of $f$.

In case $\int \Phi$ admits the global model structure, the same argument applies to trivial cofibrations, which are then characterized by the left lifting property with respect to fibrations. If $\int \Phi$ admits only the global semi-model structure, then this argument only works when the domain of $(\phi,f)$ is cofibrant.
\end{proof}

\begin{remark}
In the proof of Lemma \ref{lemma:cof-of-pairs}, we were able to deduce the trivial cofibration case from the cofibration case and the definition of weak equivalences. Similarly, in the proof above, if $(\phi,f)$ is a global trivial cofibration, we do know that $\phi$ is a horizontal trivial cofibration, even without a semi-model structure on $\int \Phi$. We also know that $f: A\to \phi^*(A')$ is a weak equivalence. But this does not tell us in general that ${}^\dashv f: \phi_!(A) \to A'$ is a weak equivalence.
\end{remark}

\begin{remark}
When both Stanculescu's model structure \cite{stanculescu} on $\int \Phi$, and our transferred model structure on $\int \Phi$ exist, they must coincide, because they have the same fibrations and weak equivalences. Similarly, thanks to Proposition \ref{prop:stanculescu}, our global model structure also agrees with the `total model structure' of \cite{cagne-mellies}, since they have the same fibrations and cofibrations. In the case when the global model structure exists, the characterization of the cofibrations in \ref{prop:stanculescu} was already known to Stanculescu. We now compare the two approaches. First, Stanculescu's requirement, that $p:\int \Phi \to \cat{B}$ be a cloven bifibration, is always true in our transferred situations. However, we do not require the existence of model structures on the base and fibers, since our approach is the opposite. Lastly, we do not require that the functors $\phi^*$ reflect weak equivalences. Our characterization of the cofibrations does not require a model structure (or even a semi-model structure) on $\int \Phi$. 
\end{remark}

\begin{remark}
It is a consequence of Proposition \ref{prop:stanculescu} that $(O,A)$ is cofibrant if and only if $O$ is cofibrant in $\cat B$ and $A$ is cofibrant in $\Phi(O)$.
\end{remark}

Having characterized the global (trivial) cofibrations, we are ready to prove our main result that will allow us to deduce vertical and horizontal (semi-)model structures from the global (semi-)model structure.

 \begin{proposition}\label{prop:globaltovertical} Suppose that 
$\Phi$ is homotopically structured and, for every $\alpha:O \to O'$ in $\cat{B}$, the functor $\alpha^*$ preserves vertical weak equivalences and fibrations.
Then
\begin{itemize} 
\item[(i)] If a vertical morphism $(id,f):(O,X) \to (O,Y)$  satisfies the left lifting property with respect to  global (trivial) fibrations then it satisfies the left lifting property with respect to vertical (trivial) fibrations.

\item[(ii)] If $(id,f):(O,X)\to (O,Y)$ is a vertical (trivial) cofibration, then it is also a global (trivial) cofibration. If $\alpha:O_1 \to O_2$ is a horizontal cofibration, then the unique induced morphism $(\alpha,a):(O_1,i_1)\to (O_2,i_2)$ is a global cofibration, where $i_1$ and $i_2$ are initial objects. If $\int \Phi$ admits the global model structure (resp. semi-model structure) then the second statement remains true for trivial cofibrations (resp. trivial cofibrations with cofibrant domain).

\item[(iii)] If for a vertical morphism $(O,A)\to (O,D)$ there exists a   global factorization into a  global cofibration followed by a global trivial fibration then such a  factorization exists vertically.

 \item[(iv)] If for a vertical morphism $(O,A)\to (O,D)$ there exists a   global factorization into a global trivial cofibration followed by a global   fibration then such a   factorization exists vertically. 
 \end{itemize}
 
 \end{proposition}

\begin{proof} 
Point (i) is trivial, since any vertical (trivial) fibration $g:A\to B$ gives rise to a global (trivial) fibration $(O,A)\to (O,B)$, and a lift in $\int \Phi$ of $(id,f)$ against $(id,g)$ is forced to be the identity on the first component and the desired lift in $\Phi(O)$ on the second component. 

For the first part of (ii), suppose $(O',E)\to (O'',B)$ is a global trivial fibration  and that we are given the following lifting problem.
\begin{align*}
\xymatrix{(O,X) \ar[d] \ar[r]^{(\phi,g)} & (O',E) \ar[d]^{(\psi,p)} \\ (O,Y)\ar[r]_{(\delta,h)} & (O'',B)}
\end{align*}
Clearly the top horizontal morphism provides the first component $O\to O'$ of our desired lift. Since $(\psi,p)$ is a trivial fibration we know that $p:E\to \psi^*(B)$ is a trivial fibration in $\Phi(O')$. Apply $\phi^*$ to this morphism and observe that $\phi^*(\psi^*(B))\cong \delta^*(B)$ because of commutativity in the square above. Furthermore, $\phi^*$ preserves trivial fibrations by assumption. So we have a lifting diagram in $\Phi(O)$
\begin{align*}
\xymatrix{X \ar[d]_f \ar[r]^g & \phi^*(E) \ar[d] \\ Y\ar[r] & \delta^*(B)}
\end{align*}
in which the left vertical morphism is a vertical cofibration (a class of morphisms defined by the lifting property) and the right vertical morphism is a trivial fibration. So this diagram admits a lift $\gamma$. The pair $(\phi,\gamma)$ is a global lift in the original diagram. In case $(O,X)\to (O,Y)$ is a vertical trivial cofibration we first prove that it is  a global cofibration as above but it is also a weak equivalence (because it is the identity on the first component and a trivial cofibration on the second component), so we are done.  

For the second part of (ii), consider the following lifting problem in $\int \Phi$, where $\alpha:O_1\to O_2$ is a horizontal cofibration and $(\psi,p)$ is a trivial fibration
\begin{align} \label{diagram-pairs-pushout}
\xymatrix{(O_1,i_1) \ar[d] \ar[r]^{(\phi,g)} & (O',E) \ar[d]^{(\psi,p)} \\ 
(O_2,i_2)\ar[r]_{(\delta,h)} & (O'',B)}
\end{align}
Since $\psi:O' \to O''$ is a horizontal trivial fibration, and horizontal cofibrations are defined by lifting, there is a lift $\beta:O_2 \to O'$. Furthermore, $\beta^*(E)$ is an object of $\Phi(O_2)$, so there is a unique morphism from the initial object $i_2$ to $\beta^*(E)$, commuting with the unique morphism $i_1 \to \phi^*(E)$ in $\Phi(O_1)$. In the case where $\alpha$ is a trivial cofibration, the same proof demonstrates that $(\alpha,a)$ is a cofibration in $\int \Phi$. Since we do not know in general that $a$ is a weak equivalence, we invoke Proposition \ref{prop:stanculescu}, and for this we require $\int \Phi$ to admit the global (semi-)model structure.

For (iii), let $(O,A)\to (O,D)$ be a vertical morphism.
We factorize it globally as
$$(O,A)\stackrel{(\phi,f)}{\to} (O',A') \stackrel{(\psi,g)}{\to} (O, D)$$
where $(\phi,f)$ is a global cofibration and $(\psi,g)$ is a global trivial fibration.
We then have a vertical factorization:
$$(O,A)\stackrel{(id,f)}{\to} (O,\phi^*(A')) \stackrel{(id,\phi^*g)}{\to} (O, D).$$
The second component is a trivial fibration by our assumption on $\phi^*$. We must prove that the first component is a vertical cofibration. Consider a lifting problem against a vertical trivial fibration $(O,X)\to (O,Y)$:

\[
\xymatrix{(O,A)\ar@{^(->}[dd] \ar[dr] \ar[rr] & & (O,X) \ar@{->>}[dd]_\simeq\\
& (O',A')  \ar@{..>}[dr] 
 & \\
 (O,\phi^*(A')) \ar[ur]   \ar[rr] & & (O,Y)}
\]

We complete this diagram to the commutative diagram as presented above, (\ref{diagram-pairs-pushout}). Here the morphism $(O,A)\to (O',A')$ is the morphism $(\phi,f)$ and the 
morphism $(O,f^*(A'))\to (O',A')$ is $(\phi, id).$  To construct the morphism $(O',A') \to (O,Y)$ we observe that we have a vertical morphism
$$(O,\phi^* (A')) \to (O,Y) = (O,\phi^*\psi^*(Y))$$
Since, $\phi^*\psi^* = Id: \Phi(O)\to \Phi(O)$, the functor $\phi^*$ is full. 
So, a morphism $\phi^*(A')\to \phi^*\psi^*(Y)$ is given by some $\phi^*(\eta)$ where $\eta:A'\to \psi^*(Y).$ 
Now we define the morphism $(O',A') \to (O,Y)$ in the diagram above to be $(g,\eta).$  The diagram commutes by construction. 

Since $(\phi,f)$ is a global cofibration we have a global lifting $(O',A')\to (O,X)$ and, hence, a vertical lifting after precomposition of the global lifting with $(\phi, id).$

The proof of (iv) works just like the proof of (iii) to prove that $(id,f)$ is a vertical cofibration. We then observe that, since $(\phi,f)$ is a trivial vertical cofibration, the morphism $(\phi,id)$ is a weak equivalence. Hence, by the two out of three property, $(id,f)$ is a vertical weak equivalence, as required.
\end{proof}

Note that Proposition \ref{prop:stanculescu}(ii) says that the functor $i: \cat B \to \int \Phi$ from (\ref{figure:adjoints}) is left Quillen (making $p$ right Quillen). Recall that the projection $p:\int \Phi \to \cat{B}$ has a right adjoint $r:\cat{B} \to \int \Phi$ taking $O \mapsto (O,t)$, where $t$ is the terminal $O$-algebra. This functor is trivially right Quillen, but happens to preserve more than one might expect, as we now show.

\begin{proposition} \label{prop:r-pres-cof}
Assume that $\Phi$ is homotopically structured. Then the functor $r$ preserves (trivial) cofibrations.
\end{proposition}

\begin{proof}
Let $\gamma: O\to P$ be a horizontal cofibration. Consider the following lifting problem against a global trivial fibration $(Q,C) \to (Z,D)$, where $(\gamma,g) = r(\gamma)$:

\[
\xymatrix{
(O,t_O) \ar[r]^{(\alpha,f)} \ar[d]_{(\gamma,g)} & (Q,C) \ar[d] \\
(P,t_P) \ar[r] & (Z,D)}
\]

Since $\gamma$ is a cofibration, we have a horizontal lift $\beta:P \to Q$, by definition of the horizontal cofibrations. We must construct a morphism $h:t_P \to \beta^*(C)$ commuting with $f$ and $g$. Since $\gamma^*$ preserves terminal objects, we have a composite $\gamma^*(t_P)\cong t_O \to \alpha^*(C) \cong \gamma^*(\beta^*(C))$. Since $\gamma^*$ is full, we have our desired morphism $t_P \to \beta^*(C)$. One triangle in the lifting diagram commutes because of the isomorphism and the other commutes because of the horizontal lift above.
This proves that $(\gamma,g)$ is a global cofibration, as that class of morphisms is defined by the lifting property.  If $\gamma$ is in addition a horizontal weak equivalence then clearly $r(\gamma) = (\gamma,g): (O,t_O)\to (P,t_P)$ is a global weak equivalence, since $g: t_O \to \gamma^*(t_P)$ is an isomorphism, since $\gamma^*$ preserves terminal objects.
\end{proof}

We are ready to derive (semi-)model structures on $\cat{B}$ and on the fibers $\Phi(O)$, assuming the existence of the global (semi-)model structure $\int \Phi$. Thanks to Theorem \ref{thm:global-model-examples}, the result below has many consequences for various flavors of operads and their algebras and modules. We will need to assume, for any $\phi:O\to O'$ in $\cat{B}$, that $\phi^*$ preserves weak equivalences. This is necessary in order that the vertical weak equivalences satisfy the two out of three property, and is satisfied in all the operadic examples.

\begin{theorem} \label{thm:global-to-local}
Assume $\Phi: \cat{B}^{op} \to CAT$ is homotopically structured, that $\int \Phi$ is complete and cocomplete, and that the categories $\Phi(O)$ are complete and cocomplete. Let also the functors $\phi^*$ preserve weak equivalences and fibrations
for every $\phi: O\to O'$ in $\cat{B}$. Then:
 
\begin{enumerate}
\item If $\int \Phi$ admits a global (semi-)model structure then $\cat{B}$ admits a horizontal (semi-) model structure, $\Phi(O)$ admits a vertical model structure for each $O\in \cat{B}$ (semi-model structure for a cofibrant $O$), and projection $p:\int \Phi \to \cat{B}$ is a right Quillen functor \cite[Definition 1.12]{barwickSemi}. 

\item For every $\phi:O\to O'$, the induced functor $\phi^*$ has a left adjoint $\phi_!$. Under our assumptions on $\phi$, $\phi^*$ and $\phi_!$ form a Quillen pair. 

\item If the global structure is cofibrantly generated then the horizontal structure is cofibrantly generated and the vertical structures $\Phi(O)$ are cofibrantly generated for every $O$ appearing as a domain in the global generating sets.
\end{enumerate}
\end{theorem}

We will see in Remark \ref{remark:u-cofibrancy-necessary} that, in the semi-model category case, (1) is false if we drop the cofibrancy assumption on $O$.
 
\begin{proof} 
We first verify the last claim of (1), that $p$ is right Quillen. If $\cat{B}$ has the horizontal (semi-)model structure and $\int \Phi$ has the global one, we see immediately that $p$ preserves (trivial) fibrations if $\cat{B}$ has a (semi-)model structure. Next, $p$ has a left adjoint, which is $i:O \mapsto (O,i_O)$, where $i_O$ is the initial object in $\Phi(O)$. We know that these initial objects exist because we assumed the categories $\Phi(O)$ are bicomplete. 

Next, we turn to the existence of the horizontal (semi)-model structure. We use the functor $i$ to see that the horizontal fibrations and weak equivalences are closed under retracts (and the two out of three property for the weak equivalences) because the same is true for the global model structure. We define the horizontal cofibrations via lifting, so that one half of lifting comes for free, as well as retract closure. To check the other lifting condition, consider the following lifting problem of a trivial cofibration $\phi$ against a fibration $\psi$:

\[
\xymatrix{O_1 \ar[r] \ar[d] & O_3 \ar[d] \\
O_2 \ar[r] & O_4}
\]

We convert this to a lifting problem in $\int \Phi$, using the initial objects $i_j$ and terminal objects $t_j$ in each $\Phi(O_j)$:

\[
\xymatrix{(O_1,i_1) \ar[r] \ar[d]_{(\phi,f)} & (O_3,t_3) \ar[d]^{(\psi,g)} \\
(O_2,i_2) \ar[r] & (O_4,t_4)}
\]

The morphism $\psi$ is a fibration by definition of the global fibrations, and $t_3 \to \psi^*(t_4)$ is an isomorphism (recall that, because $\int \Phi$ is cocomplete, functors $\psi^*$ are right adjoints and hence preserve terminal objects). The left vertical morphism is a trivial cofibration by Proposition \ref{prop:globaltovertical} (ii). Hence, there is a lift in $\int \Phi$, when $\int \Phi$ has a full model structure, and the first component of this lift gives the desired horizontal lift $O_2 \to O_3$. When $\int \Phi$ has only a semi-model structure, this only works if $O_1$ is cofibrant, so that $(O_1,i_1)$ is cofibrant. 

We turn to factorization. Let $\alpha:O_1\to O_2$ be a morphism in $\cat{B}$. Factor $i(\alpha):(O_1,i_1)\to (O_2,i_2)$ in $\int \Phi$ into a cofibration $(\phi,f):(O_1,i_1)\to (P,A)$ followed by a trivial fibration $(\psi,g)$. By definition, the morphism $\psi:P\to O_2$ is a trivial fibration. By Proposition \ref{prop:stanculescu}, $\phi$ is a cofibration. So half of the factorization axiom is done.

The case for factoring into a trivial cofibration followed by a fibration is similar, but when $\int \Phi$ is only a semi-model category, the factorization of $i(\alpha)$ in $\int \Phi$ only works when $O_1$ is cofibrant. We will see shortly that when $\int \Phi$ is a semi-model category over some model category $\M$ (e.g., the base $\cat B$), then this argument works when $U(O_1)$ is cofibrant in $\M$.
The other factorization works similarly, but requires no cofibrancy assumption in the semi-model category case.

We turn to the vertical model structure. The retract axiom and the two out of three axiom are automatically satisfied, as above. 
Suppose $\int \Phi$ has a full model structure. Using Proposition \ref{prop:globaltovertical} (i), (ii), and the lifting property in $\int \Phi$, we deduce both vertical lifting axioms. Similarly, using Proposition \ref{prop:globaltovertical} (iii) and (iv), we deduce the vertical factorization axioms. 

If $\int \Phi$ has only a semi-model structure, then the axioms regarding cofibrations and trivial fibrations work as usual, but axioms with trivial cofibrations only work for morphisms with cofibrant domain. Let $f:X\to Y$ be a morphism in $\Phi(O)$ with cofibrant domain. We only know that $(id,f)$ factors as a trivial cofibration followed by a fibration if $O$ is cofibrant. In this case, Proposition \ref{prop:globaltovertical} (iv) gives the desired vertical factorization. If $f$ is a trivial cofibration, in order for $(id,f):(O,X)\to (O,Y)$ to lift against fibrations, we must know that its domain is cofibrant, i.e. that $O$ is cofibrant in $\cat{B}$ and $X$ is cofibrant in $\Phi(O)$. Additionally, we must argue that $(id,f)$ is a global trivial cofibration. The proof proceeds just like (ii) of Proposition \ref{prop:globaltovertical}, since trivial cofibrations \textit{whose domain is cofibrant} are characterized by a lifting property \cite[Lemma 1.7]{barwickSemi}.

If $\int \Phi$ is a semi-model category \textit{over} some model category $\M$, then the axioms regarding trivial cofibrations hold for morphisms whose domain becomes cofibrant in $\M$. Hence, it is sufficient for $U(O)$ to be cofibrant in $\M$. Such $O$ will have factorization and lifting in $\int \Phi$, using that trivial cofibrations with such domains are characterized by a lifting property in any semi-model category over $\M$. In the context of operads, this is saying that when $O$ is $\Sigma$-cofibrant then there is a semi-model structure on $O$-algebras.

We turn to the proof of (2). We have already observed that, for any $\phi:O\to O'$, in $\cat{B}$, $\phi^*$ has a left adjoint, since $\int \Phi$ is cocomplete. Under our hypotheses, $\phi^*$ will preserve the (trivial) fibrations we have defined, hence is a right Quillen functor (even if $\Phi(O)$ and $\Phi(O')$ have only semi-model structures) \cite[Definition 1.12]{barwickSemi}. 

For (3), suppose $\int \Phi$ is cofibrantly generated by sets of morphisms $I_{GR}$ and $J_{GR}$. We define horizontal generators $I_H$ (resp. $J_H$) as morphisms $\phi:O\to O'$ such that there is a morphism in $I_{GR}$ (resp. $J_{GR}$) of the form $(\phi,f): (O,A)\to (O',A')$. Smallness of the horizontal domains follows immediately from smallness of the global domains. Next, we must check lifting of $\phi: O\to O'$ against any trivial fibration (resp. fibration) $\psi: P \to Q$. We embed the lifting problem into $\int \Phi$, using the embedding $P \mapsto (P,t_P)$ where $t_P$ is terminal in $\Phi(P)$.

\begin{align} \label{diagram:cofgen}
\xymatrix{
(O,A) \ar[d] \ar[r] & (P,t_P) \ar[d]\\
(O',A')\ar[r] & (Q,t_Q)\\
}
\end{align}

Since the left vertical morphism is a generating (trivial) cofibration, we have a lift. The first component of this lift is the required horizontal lift, proving that all horizontal (trivial) fibrations satisfy the lifting property with respect to $J_H$ (resp. $I_H$). 
Conversely, if $\psi:P\to Q$ is any horizontal morphism satisfying the lifting property with respect to all morphisms in $J_H$ (resp. $I_H$), then the induced global lift (using that $t_P$ and $t_Q$ are terminal) in any diagram like (\ref{diagram:cofgen}) shows that $(P,t_P)\to (Q,t_Q)$ is a (trivial) fibration. It follows that $P\to Q$ is a horizontal (trivial) fibration, completing the proof that $I_H$ and $J_H$ generate the horizontal structure. 

For the vertical structure on $\Phi(O)$, observe that $\int \Phi \downarrow (O,t_O)$ is cofibrantly generated by morphisms:

\[
\xymatrix{
(P,A)\ar[rr]^{(\phi,f)} \ar[dr]_\alpha && (Q,B)\ar[dl]^\beta \\
& (O,t_O) & 
}\]

where $(P,A)\to (Q,B)$ is in $I_{GR}$ (resp. $J_{GR}$). This is a set because there is only a set worth of choices for $\alpha$ and $\beta$. Define the set $I_V^O$ (resp. $J_V^O$) to be the set of morphisms of the form $\beta_! \phi_!(A)\to \beta_!(B)$ coming from the generators for $\int \Phi \downarrow (O,t_O)$. By Proposition \ref{prop:stanculescu}, the morphisms $\phi_!(A)\to B$ are (trivial) cofibrations. Since $\beta_!$ is left Quillen, so are the morphisms in $I_V^O$ (resp. $J_V^O$). We now show these morphisms generate.

Consider any morphism $g:X\to Y$ in $\Phi(O)$ that has the right lifting property with respect to all morphisms in $I_V^O$. To show $g$ is a vertical trivial fibration, we set up a lifting problem in $\int \Phi$ (displaying only the horizontal morphisms):

\begin{align} \label{diagram:cofgen2}
\xymatrix{
(P,A) \ar[d]_\phi \ar[r]^\alpha & (O,X) \ar[d]\\
(Q,B)\ar[r]_\beta & (O,Y)\\
}
\end{align}

Since the horizontal morphism on the right is the identity, there is a lift $\beta:Q \to O$. In order to get a vertical lift, we need a morphism $B\to \beta^*(X)$ in $\Phi(Q)$. Equivalently, we need a morphism $h:\beta_!(B)\to X$ in $\Phi(O)$. By definition of $I_V^O$, we have a lift in the following diagram in $\Phi(O)$:

\begin{align} \label{diagram:cofgen3}
\xymatrix{
\beta_! \phi_! A \ar[r] \ar[d] & X \ar[d] \\
\beta_! B \ar[r] & Y
}
\end{align}

This gives the required morphism $h$. The two triangles in (\ref{diagram:cofgen2}) commute because the triangles in (\ref{diagram:cofgen3}) commute. For the upper triangle, use that $\beta_! \phi_!A \cong \alpha_! A$
\end{proof}

\begin{remark}
Most often, our global (semi-)model structure will be encoded by a polynomial monad and transferred from some cofibrantly generated model category $\M$, as detailed in \cite[Section 3]{companion}. In such a case, $\int \Phi$ will be cofibrantly generated, and will be combinatorial if $\M$ is. Theorem \ref{thm:global-to-local} proves the same for the horizontal and vertical model structures, even if they are not transferred. For example, in Section \ref{sec:twisted} we will see an application producing a model structure on twisted modular operads (Theorem \ref{thm:twisted-modular}) that is not transferred, but Theorem \ref{thm:global-to-local} still applies.

\end{remark}

We have seen that, if $\int \Phi$ is cofibrantly generated, then the same is true for the base $\cat B$ and the fibers $\Phi(O)$. We now provide a partial converse to this statement. However, since there is a class worth of fibers $\Phi(O)$, we will not be able to deduce that $\int \Phi$ is cofibrantly generated by a {\em set} of morphisms. 

\begin{proposition} \label{proposition:class-cof-gen}
Suppose $\Phi$ is homotopically structured and that structure makes $\int \Phi$ into a (semi-)model category. Suppose the (semi-)model categories (from Theorem \ref{thm:global-to-local}) $\cat{B}$ and $\int \Phi$ are all cofibrantly generated. Then a morphism $(\phi,f)$ in $\int \Phi$ is a fibration (resp. trivial fibration) if and only if it satisfies the right lifting property with respect to all morphisms $(\psi,g): (O,A)\to (P,B)$ where $\psi$ is a generating trivial cofibration (resp. generating cofibration) in $\cat B$, and $g: A\to \psi^*(B)$ is a generating trivial cofibration (resp. generating cofibration) in $\Phi(O)$.
\end{proposition}

Below, let $i_O$ denote the initial object in $\Phi(O)$, and similarly for $P$.

\begin{proof}
Given sets $I_H$ of horizontal generating cofibrations, and $I_V^O$ vertical generating cofibrations in $\Phi(O)$ for all $O$, define the collection $I_{GR}$ to consist of all morphisms $(O,i_O)\to (P,i_P)$ where $O\to P$ is in $I_H$, together with morphisms $(O,A)\to (O,B)$ where $A\to B$ is in $I_V^O$. Observe that, because there is a class worth of objects $O \in \int \Phi$, this latter collection of morphisms is a class. Define $J_{GR}$ similarly via $J_H$ and $J_V^O$. 

Proposition \ref{prop:globaltovertical} shows that the morphisms in $I_{GR}$ are global cofibrations (trivial for $J_{GR}$). To prove they generate, let $(\phi,f):(Q,C)\to (Z,D)$ have the right lifting property with respect to $I_{GR}$. Applying projection $p$, we see that $Q\to Z$ has the right lifting property with respect to every morphism in $I_H$, so it's a trivial fibration. Consider $f:C\to \phi^*(D)$ in $\Phi(Q)$. By definition of $I_{GR}$, this morphism must have the lifting property with respect to each morphism in $I_V^Q$, so must be a trivial fibration. It follows that $(\phi,f)$ is a trivial fibration. The argument for $J_{GR}$ is similar.
\end{proof}

In many of our situations of interest, the Grothendieck construction will have a transferred (semi-)model structure. For example, if $\cat B$ is a category of operads (e.g. symmetric, non-symmetric, reduced, etc.) then $\cat B$ often has a transferred (semi-)model structure \cite{BM03, fresse-book} coming from an appropriate category $\cat N$ of collections built as an infinite product of model structures coming from some base model category $\M$. In these settings, the Grothendieck construction, for pairs $(O,A)$ where $O$ is an operad and $A$ is an $O$-algebra, is transferred from $\cat N \times \M$ (or from $\cat N \times \M^\fC$ for $\fC$-colored operads). Furthermore, sometimes the fibers $\Phi(O)$ have a transferred (semi-)model structure from the base model category $\M$ (or from $\M^\fC$). We now set up a general framework to address such transfer situations.

Given two functors $\Psi:{\cat N}^{op} \to CAT$ and $\Phi: \cat B^{op} \to CAT$, we define a forgetful morphism to be a pair $(U,u)$ consisting of a forgetful functor $U:\cat{B}^{op} \to \cat{N}^{op}$, and a $2$-natural transformation $u: \Phi \to \Psi \circ U$. If $\Phi$ is homotopically structured then we define a horizontal weak equivalence (fibration) as $f$ such that $U(f)$ is a horizontal weak equivalence (fibration). Similarly $\phi$ is a vertical weak equivalence (fibration) if $u(\phi)$ is such. With these definitions, the corresponding global  structure on $\int \Phi$ will be called {\em transferred} along $(U,u).$

\begin{theorem} \label{thm:global-transfer}
Let $(U,u)$ be a forgetful morphism from $\Phi$ to $\Psi$, and that both functors are homotopically structured. Suppose, for every $O\in \cat B$, that $\Psi(U(O))$ is a cofibrantly generated model category and that there is an adjunction $O(-):\Psi(U(O)) \leftrightarrows \Phi(O):U$, where $U$ is the right adjoint. 

\begin{enumerate}
\item If for an object $O\in \cat{B}$ any vertical morphism $(O,A)\to (O,D)$ admits the (trivial cofibration, fibration) factorization then the category $\Phi(O)$ admits a vertical model structure transferred from $\Psi(U(O))$;
\item If the global model structure exists then it coincides with the integral model structure of \cite{harpaz-prasma-integrated}.
\end{enumerate}
\end{theorem}

\begin{proof}
To prove (1), we follow the proof of Theorem 2.9 in \cite{spitzweck-thesis}, and define the weak equivalences (resp. fibrations) as morphisms which become so in $\Psi(U(O))$. 
By adjointness, the fibrations are $O(J)$-inj and the trivial fibrations are $O(I)$-inj. Define the cofibrations to be $O(I)$-cof. The two out of three property and retract axioms for fibrations and weak equivalences follow from these axioms in $\Psi(U(O))$. That cofibrations are closed under retracts is part of the definition of $O(I)$-cof. The lifting property for cofibrations against trivial fibrations follows from the definition of these classes of morphisms. Factorization into a cofibration followed by a trivial fibration follows from the small object argument and the definition of these classes of morphisms. 

Proposition \ref{prop:globaltovertical} proves that every morphism in $\Phi(O)$ has a factorization into a trivial cofibration followed by a fibration. In order to prove that trivial cofibrations lift against fibrations, we prove that $O(J)$-cell is contained in the trivial cofibrations. Let $f$ be a morphism in $O(J)$-cell and factor $f$ as $p\circ i$ where $p$ is a trivial fibration and $i$ is a cofibration. By definition of the class of trivial fibrations as $O(J)$-inj, $f$ lifts against $p$. So $f$ is a retract of $i$ by the retract argument, hence $f$ is a cofibration. To show $f$ is additionally a weak equivalence, proceed as in Lemma \ref{lemma:factorization-suffices}. Now let $g$ be a trivial cofibration. We will show that $g$ has the left lifting property against fibrations. Use the small object argument to factor $g$ into $p\circ i$ with $i \in O(J)$-cell and $p$ a fibration (by definition of fibrations as $O(J)$-inj). Since $g$ has the left lifting property with respect to $p$, $g$ is a retract of $i$, hence has the left lifting property with respect to fibrations, just like $i$.

At this point, we have proven that $\Phi(O)$ is a semi-model category. Lemma \ref{lemma:factorization-suffices} then proves that $\Phi(O)$ has a full model structure.

For (2), observe that the model structure on $\int \Phi$ is transferred along the adjunction from $\M$. This was noted in a special case in Remark 6.3.17 of \cite{harpaz-prasma-integrated}, and is clear from Definition 3.0.4 in \cite{harpaz-prasma-integrated}. Note that, in this special case, the cofibrations in $\int \Phi$ are morphisms $(O,A)\to (P,B)$ that consist of a cofibration in $\cat{B}$ and a cofibration in $\Phi(O)$.
\end{proof}

\begin{remark}
With this theorem in hand, we can compare our approach with that of Harpaz and Prasma. In \cite{harpaz-prasma-integrated}, the authors assume that the base and fibers have full model structures. In addition, they must assume $\Phi$ is relative, i.e. takes horizontal weak equivalences to Quillen equivalences. For us, this is like asking $\int \Phi$ to be left proper, as we will see shortly. Lastly, Harpaz and Prasma require their functors to be ``proper" so that left adjoints $\alpha_!$ preserve weak equivalences. We never need this condition. Our primary method for placing a (semi-)model structure on $\int \Phi$ is to transfer it. This means not every Grothendieck construction can be studied using our methods. For example, Harpaz and Prasma study the classical Grothendieck construction for comma categories. This is not transferred. Harpaz and Prasma study Grothendieck constructions for monoids and modules, and for commutative monoids and modules. These settings are transferred and our results improve on Harpaz and Prasma by removing several conditions. Lastly, Harpaz and Prasma  study a Grothendieck construction for global equivariant homotopy theory (where the objects are $(G,X)$ where $G$ is a simplicial group and $X$ is  $G$-space). The model structure on simplicial groups is transferred, as are the model structures on $G$-spaces. Thus, this example falls under our setting as well. Harpaz and Prasma also consider the left induced model structure on $G$-spaces, where weak equivalences and cofibrations are created (and preserved) by the forgetful functor to $Top$. This suggests that there might be value in generalizing our approach to allow for left induced transfers, or mixed transfers (right induced on one component, and left induced on the other). We leave this as a problem for the interested reader.
\end{remark}

We conclude with a dual to Proposition \ref{prop:r-pres-cof}, which we need later. We have been unable to prove this result in the same generality as Proposition \ref{prop:r-pres-cof}. To prove that the functor $i$ preserves (trivial) fibrations, we have had to assume the global semi-model structure exists. 

\begin{proposition} \label{prop:i-right-Quillen}
Assume $\Phi$ is homotopically structured and that the global semi-model structure exists and is cofibrantly generated.
Then $i: \cat B \to \int \Phi$ preserves (trivial) fibrations.
\end{proposition}

\begin{proof}
Let $\nu: Q\to Z$ be a trivial fibration in $\cat B$. We must prove that $i(\nu)=(\nu,i): (Q,i_Q)\to (Z,i_Z)$ is a global trivial fibration. That means, we must prove that $i_Q \to \nu^*(i_Z)$ is a trivial fibration in $\Phi(Q)$. Even though cofibrations are defined by the left lifting property, we do not know a priori that trivial fibrations are characterized by the right lifting property. The normal proof of this requires factorizations to exist \cite[Lemma 1.1.10]{hovey-book}. Fortunately, when $\int \Phi$ has a cofibrantly generated global semi-model structure, then fibrations and trivial fibrations are characterized by lifting \cite[Lemma 1.7]{barwickSemi}.

Consider the following global lifting problem, where $(\gamma,g)$ is a cofibration:
\[
\xymatrix{
(O,A) \ar[r]^{(\alpha,f)} \ar[d]_{(\gamma,g)} & (Q,i_Q) \ar[d]^{(\nu,i)} \\
(P,B) \ar[r]_{(\delta,q)} & (Z,i_Z)}
\]
To prove $i(v) = (v,i)$ is a trivial fibration, it suffices to construct a global lift. As in Proposition \ref{prop:r-pres-cof}, we have a horizontal lift $\beta: P\to Q$, and must construct the lift $h:B \to \beta^*(i_Q)$. Using left adjoints, we consider the following lifting problem in $\Phi(Z)$:
\[
\xymatrix{
\delta_! \gamma_! A = \nu_! \alpha_! A \ar[r] \ar[d] & \nu_! i_Q \ar[d] \\
\nu_! \beta_! B = \delta_! B \ar[r] & i_Z}
\]
where all morphisms are induced by the global diagram above. Since $\nu_!(i_Q) \cong i_Z$, this diagram admits a lift $\tilde{h}:\nu_! \beta_! B\to \nu_! i_Q$. Since $\nu_!$ is full, we have the morphism $h$ we wanted. Commutativity of one triangle follows from the isomorphism and for the other triangle it is a consequence of the horizontal lift above. Since $i(\nu)$ has the right lifting property with respect to all cofibrations (in the semi-model case, it is enough to have the right lifting property with respect to generating cofibrations), it is a trivial fibration as required. The case for $i$ preserving fibrations is analogous.
\end{proof}

\subsection{Relative global structure}

Often, categories of operads and operad-algebras admit only a transferred semi-model structure \cite{fresse-book}. The situation is especially nice for algebras over a $\Sigma$-cofibrant colored colored operad (including as an example most categories of operads) \cite{gutierrez-rondigs-spitzweck-ostvaer}. In the setting of a forgetful morphism $(U,u): \Phi \to \Psi$, we can speak about relative global, vertical, or horizontal cofibrant objects, and can therefore generalize the notion of $\Sigma$-cofibrancy and relative left properness \cite{batanin-berger}. For the following definition, recall that we can define cofibrancy (via lifting) regardless of whether or not the classes of weak equivalences and fibrations, provided by a homotopical structuring, satisfy the axioms for a model structure.

\begin{defn} \label{defn:u-cofibrant}
Let $\Phi: \cat B^{op}\to CAT$ and $\Psi: \cat N^{op}\to CAT$ be homotopically structured. Given a forgetful morphism $(U,u): \Phi \to \Psi$, an object $O\in \cat B$ is called $u$-cofibrant if $U(O)$ in cofibrant in $\cat N$.
\end{defn}

\begin{proposition} \label{prop:global-to-vert-over}
Let $\Phi, \Psi$, and $(U,u)$ be as in Definition \ref{defn:u-cofibrant}, and let $\Phi$ be as in Theorem \ref{thm:global-to-local}, such that $\int \Phi$ has a transferred semi-model structure over $\int \Psi$ (Definition \ref{defn:semi-model-over}). Then, $\cat B$ has a transferred semi-model structure over $\cat N$, and for each $u$-cofibrant $O\in \cat B$, there is a semi-model structure on $\Phi(O)$ over $\Psi(U(O))$.
\end{proposition}

\begin{proof}
The proof proceeds just like Theorem \ref{thm:global-to-local} (1). The point is that, because $\int \Phi$ has a semi-model structure over $\int \Psi$, we have lifting of trivial cofibrations whose domain $(O,A)$ becomes a cofibrant object $u(O,A)$ in $\int \Psi$. As a result, $O$ is only required to satisfy the property that $U(O)$ is cofibrant in $\cat N$. Similarly, for the vertical semi-model structures on $\Phi(O)$, where $u(O)$ is cofibrant in $\cat N$, any morphism $f:A\to D$ in $\Phi(O)$ gives rise to a global morphism $(O,A)\to (O,D)$. If $A$ is cofibrant in $\Psi(U(O))$ then $(O,A)$ is cofibrant in $\int \Psi$ and so $(id,f)$ admits a global factorization. Proposition \ref{prop:globaltovertical} then guarantees the existence of the vertical semi-model structure over $\Psi(U(O))$.
\end{proof}

\begin{remark} \label{remark:u-cofibrancy-necessary}
The $u$-cofibrancy condition in the preceding proposition is necessary. If $\Phi$ assigns to each symmetric operad $O$, its category of $O$-algebras, then $\int \Phi$ has a global semi-model structure, but not every category of $O$-algebras has a transferred semi-model structure. For example, if $O = Com$ and $\M$ is the projective model structure on chain complexes over a field of characteristic 2, then $Com$-algebras do not admit a transferred semi-model structure.
\end{remark}

\subsection{Properness}

Our next theorem concerns properness of the vertical and horizontal model structures coming out of the global model structure. Observe that the notion of properness (left or right) does not depend on the existence of a full model structure. 
If the homotopical structure on $\Phi$ is obtained as a transfer along a forgetful morphism $(U,u): \Phi \to \Psi$ it also makes sense to speak about relative left properness \cite{batanin-berger}. 

\begin{theorem} \label{thm:properness}
Let $\Phi$ be homotopically structured, and suppose there exist global and vertical pushouts in $\int \Phi$. If $\int \Phi$ is left proper then $\Phi(O)$ is left proper for any $O.$ If, in addition, there exists an initial object $i_O$ in each $\Phi(O)$, then $\cat{B}$ is left proper as well. 

Dually, suppose there exist global and vertical pullbacks in $\int \Phi$, and $\phi^*$ is a right adjoint for every morphism $\phi$ in $\cat B$. If $\int \Phi$ is right proper then $\Phi(O)$ is right proper for any $O.$ If, in addition, there exists a terminal object $t_O$ in $\Phi(O)$ then $\cat{B}$ is right proper as well.  
\end{theorem}

\begin{proof}
Let $w:(O,X)\to (O,Y)$ be a vertical weak equivalence and $q:(O,X)\to (O,A)$ a vertical cofibration.  Form the following pushout diagram:
\begin{align*}
\xymatrix{X \ar[r]^w \ar[d]_q \po & Y\ar[d]^\beta \\ A \ar[r]_\alpha & P}
\end{align*}

We must show the morphism $\alpha:A\to P$ in the pushout diagram is a weak equivalence. In order to use the fact that $\int \Phi$ is left proper, consider the following pushout diagram in $\int \Phi$:
\begin{align} \label{diagram-pushout-pairs}
\xymatrix{(O,X) \ar[r]^{(1,w)} \ar[d]_{(1,q)} \po & (O,Y)\ar[d]^{(\phi_1,h)} \ar@/^1pc/[ddr]^{(1,\beta)} &
 \\ (O,A) \ar[r]_{(\phi_2,f)} \ar@/_1pc/[drr]_{(1,\alpha)} & (O',A') \ar@{..>}[dr]^{(\psi,g)} & \\
& & (O,P)}
\end{align}

First, note that commutativity of the square implies $\phi_1 \circ 1 = \phi_2 \circ 1$ so $\phi_1 = \phi_2$ and we will denote this morphism $\phi$ from now on. By definition of $P$, the pushout diagram morphisms to $(O,P)$, and the morphism on first components is the identity. This induces a morphism $(\psi,g):(O',A')\to (O,P)$ with the property that $\psi \circ \phi = id_O$. This implies $\phi^*(\psi^*(Z))\cong Z$ for all  $Z$.

We claim that $P \cong \phi^*(A')$. Because $\int \Phi$ is left proper, the morphism $(\phi,f)$ is a global weak equivalence and so $f:A\to \phi^*(A')\cong P$ is a vertical weak equivalence  as required.

To verify this claim we use the universal property of pushouts to exhibit inverse morphisms going both ways. First, $A$ and $Y$ both map to $\phi^*(A').$  This induces a morphism $\gamma:P\to \phi^*(A')$ by the universal property of $P$.

Next, the morphism  $(\psi,g):(O',A')\to (O,P)$ produces a morphism $\phi^*(g):\phi^*(A')\to \phi^*(\psi^*(P))\cong P$. 

We now show that $\phi^*(A')$ satisfies the universal property of the vertical pushout, so that it must be isomorphic to $P$. Observe from (\ref{diagram-pushout-pairs}) that $\beta = \phi^*(g)\circ h$ and $\alpha = \phi^*(g)\circ f$, by commutativity of the two curved triangles. For any   $Z$ that fits into a diagram of the form:

\begin{align*}
\xymatrix{X \ar[r]^w \ar[d]_q \po & Y \cong \phi^*(\psi^*(Y)) \ar[d]^\beta \ar@/^1pc/[ddr] &
\\ A \cong \phi^*(\psi^*(A)) \ar[r]_\alpha \ar@/_1pc/[drr] & P \ar[dr] & \\
& & Z}
\end{align*}

the induced morphism $P\to Z$ can be written $\phi^*(\psi^*(P))\to Z$, and hence there is a morphism $\phi^*(A')\to \phi^*(\psi^*(P))\to Z$ coming from precomposition with $\phi^*(g)$. Furthermore, because $\beta = \phi^*(g)\circ h$ and $\alpha = \phi^*(g)\circ f$, the following diagram commutes

\begin{align*}
\xymatrix{X \ar[r]^w \ar[d]_q \po & Y \ar[d]^\beta \ar@/^1pc/[ddr] &
\\ A \ar[r]_\alpha \ar@/_1pc/[drr] & \phi^*(A')\ar[dr] & \\
& & Z}
\end{align*}

This proves the universal property for $\phi^*(A')$ and so we see that $\phi^*(A')\cong P$ as required.

The claims for right properness is analogous but we need an additional fact that $\phi^*$ preserves vertical pullbacks since it is a right adjoint. We leave this as an exercise. 
The claims about properness of $\cat{B}$ are even easier.
\end{proof}

Right properness is important for the theory of right Bousfield localization, which is further explored in operadic contexts in \cite{white-yau2,white-yau4}. Similarly, left properness is important for left Bousfield localization, further explored in \cite{batanin-white-eilenberg-moore, white-localization, white-yau, white-yau-co}.

\begin{theorem} \label{thm:rel-left-proper}
Suppose $\Phi$ and $\Psi$ are homotopically structured, and $\int \Phi$ is relatively left proper over $\int \Psi$ (Definition \ref{defn:relatively-left-proper}), along a forgetful morphism $(U,u): \Phi\to \Psi$. Let $O$ be $u$-cofibrant. Then $\Phi(O)$ is relatively left proper over $\Psi(U(O))$, and $\cat B$ is relatively left proper over $\cat N$, if every $\Phi(O)$ and $\Psi(O)$ has an initial object. The dual statements for right properness also hold.
\end{theorem}

\begin{proof}
Suppose we are given a pushout diagram in $\Phi(O)$,
\begin{align*}
\xymatrix{X \ar[r]^w \ar[d]_q \po & Y\ar[d]^\beta \\ A \ar[r]_\alpha & P}
\end{align*}
in which $U(X)$ and $U(Y)$ are cofibrant. We wish to show $A\to P$ is a weak equivalence, so we write down the corresponding pushout diagram in $\int \Phi$:
\begin{align}
\xymatrix{(O,X) \ar[r]^{(1,w)} \ar[d]_{(1,q)} \po & (O,Y)\ar[d]^{(\phi_1,h)}\\
 (O,A) \ar[r]_{(\phi_2,f)} & (O',A') }
\end{align}

In order to apply our hypothesis that $\int \Phi$ is relatively left proper we have to know that $U((O,X))$ and $U((O,Y))$ are cofibrant in $\int \Psi$, 
which follows from the $u$-cofibrancy of $O$ and the assumption about $X$ and $Y$.
Since $\int \Phi$ is relatively left proper, the morphism $(\phi_2,f)$ across the bottom of (\ref{diagram-pushout-pairs}) is a weak equivalence. The same proof as in the left proper case proves that $A\to P$ is a vertical weak equivalence. 
The remaining statements are similar, and are left as an exercise for the reader.
\end{proof}

We conclude with an application that, to the best of our knowledge, has not appeared previously in the literature. We adopt the notation $Gr(SOp,alg)$ for the polynomial monad whose algebras are pairs $(O,A)$ where $O$ is a symmetric operad and $A$ is an $O$-algebra. Similarly, $Gr(SOp,mod)$ is for the case where $A$ is a left $O$-module, and $Gr(NOp,\ast)$ is for non-symmetric operads (where $\ast$ can be either $alg$ or $mod$). All four of these monads yield global model structures over monads $\int \Psi$ that describe the relevant categories of collections (i.e., symmetric or non-symmetric) and the fibers.

\begin{corollary}
Fix a field $k$ of characteristic zero and a set $\fC$ of colors. Let $\M$ be the projective model structure on (bounded or unbounded) chain complexes over $k$. Then the global model structure $\int \Phi$ is left proper (for any of the four cases above), as are the model structures on $\fC$-colored operads (symmetric or non-symmetric) and on $O$-algebras and $O$-modules for any $\fC$-colored operad $O$.
\end{corollary}

\begin{proof}
The key fact is that every symmetric sequence in $\M$ is projectively cofibrant. This essentially follows from Maschke's theorem, and is explained in \cite[Corollary 5.2.4]{white-yau6}. One now mimics the proof that the category of symmetric operads is relatively left proper \cite{hackney-proper}. This reduces the question to the analysis of a pushout of a span of the form $A \gets F(K) \to F(L)$ where $F$ is the the free $T$-algebra functor, $A$ is a $T$-algebra, and the morphism $i:K\to L$ is a (trivial) cofibration in the model category $\int \Psi$ from which $\int \Phi$ is transferred. To prove relative left properness, the authors of \cite{hackney-proper} reduce to the case where $A$ is $u$-cofibrant. Since every polynomial monad is given by a $\Sigma$-cofibrant colored operad (which we denote by $P$ for the purposes of this proof), we can sidestep the filtration via trees of \cite{hackney-proper}, and instead use the filtration of \cite{white-yau} (5.3.2), which relies on the enveloping operad $P_A$.

Because every symmetric sequence is projectively cofibrant, $P_A\duc$ is $\Sigmac$-projectively cofibrant, meaning at every step of the filtration we are taking the pushout of a (trivial) cofibration in $\M$, since $i$ is a (trivial) cofibration. In the language of \cite{batanin-berger}, this proves that each of the monads in question is $K$-adequate for the class $K$ of cofibrations. Now the proof of \cite[Lemma 3.1.7]{hackney-proper} (or, equivalently, \cite[Theorem 2.14]{batanin-berger}) goes through directly. All the morphisms denoted $h$ in \cite[Lemma 3.1.7]{hackney-proper} are cofibrations and all the morphisms denoted $f$ are weak equivalences, since all objects in $\M$ are cofibrant. We conclude that $\int \Phi$ is relatively left proper. However, because every $A$ is $u$-cofibrant (again, since all objects in $\M$ are cofibrant), we see that $\int \Phi$ is left proper. For the corresponding statements about horizontal model structures (that is, categories of operads) and vertical model structures (categories of $O$-algebras and $O$-modules) we appeal to Theorem \ref{thm:properness}.
\end{proof}

\begin{remark}
The main reason one would like the global model structure $\int \Phi$ to be left proper is to left Bousfield localize it. We have proven a result connecting global localizations with horizontal and vertical localizations, and plan to include that result in a future work with several planned applications of global localizations.
\end{remark}

\subsection{Rectification}

A common problem in homotopy theory is to lift Quillen equivalences of model categories to Quillen equivalences of $T$-algebras, for suitable monads $T$. For example, if $T$ is given by a colored symmetric operad, then a general solution to this problem (as well as several applications and examples) is provided in \cite{white-yau3}.  A more general approach to lifting Quillen equivalences goes via homotopy Beck-Chevalley squares \cite[Theorem 4.2.2]{Reedy-paper}.

In this section, we study the question from the point of view of the global model structure. We shall determine when a weak equivalence $O\to O'$ in $\cat B$ induces a Quillen equivalence on vertical (semi-)model structures $\Phi(O)$ and $\Phi(O')$. Throughout, assume that $\Phi$ and $\Psi$ are homotopically structured and that $(U,u): \Phi\to \Psi$ is a forgetful morphism. The following theorem has already been applied in the PhD thesis of the second author.

\begin{theorem}\label{theoremquillenequivfibers}
Let $\int \Phi$ be homotopically structured.

\begin{enumerate}
\item If $\int \Phi$ admits the global model structure, and it's left proper, then each $\phi^*$ has a Quillen left adjoint. For a horizontal weak equivalence $\phi$, the unit of the adjunction $(\phi^*, \phi_!)$ is a weak equivalence. Moreover, if $\phi^*$ reflects weak equivalences and the unique morphism $\tau: i_O \to \phi^*(i_{O'})$ is a weak equivalence then $(\phi^*,\phi_!)$ is a Quillen equivalence. 

\item Suppose $\int \Phi$ admits a global model structure or semi-model structure over $\int \Psi$ transferred along $(U,u):\Phi \to \Psi,$ that is cofibrantly generated, and relatively left proper. Assume $\Phi(O)$ and $\Phi(O')$ admit vertical model or semi-model structures over $\Psi(U(O))$ and $\Psi(U(O'))$, respectively. Assume $O$ and $O'$ are relatively cofibrant. Then any $\phi:O \to O'$ induces a Quillen pair $(\phi^*,\phi_!)$. Furthermore, if $\phi$ is a horizontal weak equivalence, and $\phi^*$ reflects weak equivalences, then $(\phi^*, \phi_!)$ is a Quillen equivalence of semi-model categories.

\item The conclusion of (2) remains valid if $\int \Phi$ admits a cofibrantly generated global semi-model structure and $O$ and $O'$ are cofibrant in $\int \Phi$.
\end{enumerate}
\end{theorem}

\begin{proof} We begin with (1). Consider the following pushout in $\int \Phi$, where $A$ is cofibrant, so the left vertical morphism is a cofibration.

\begin{align*}
\xymatrix{(O,i_O) \ar[r]^{(\phi,\tau)} \ar[d] \po & (O',i_{O'}) \ar[d] \\ (O,A) \ar[r]_{(\phi,\epsilon)} & (O',P)}
\end{align*}
This pushout can be computed using Lemma \ref{lemma:pushout in Gr}. It is not hard to check that $P$ has the universal property of $\phi_!(A)$ and that this procedure defines a left adjoint to $\phi^*$.
Furthermore, $\phi_!$ preserves (trivial) cofibrations, since, for a cofibration $g: A\to B$ in $\Phi(O)$, $\phi_!(g)$ is obtained as the colimit of a projective cofibration in the diagram category $(\int \Phi)^D$ where $D$ is the span category $\bullet \gets \bullet \to \bullet$. The morphism of diagrams in question is a projective cofibration because it is an objectwise cofibration, because $(O,i_O)\to (O,A)$ and $(O,i_O)\to (O,B)$ are cofibrations, and because $\int \Phi$ is left proper.

If $\phi$ is a horizontal weak equivalence, then the unit $\eta: A\to \phi^*P \cong \phi^*\phi_! A$ is a weak equivalence, by the left properness of $\int \Phi$. Now suppose $\phi^*: \Phi(O') \to \Phi(O)$ reflects weak equivalences. Then the counit $\epsilon: \phi_! \phi^* B\to B$ is a weak equivalence, because the pushout $(O,\phi^*(B))\to (O',\phi_! \phi^*(B))$, of $(\phi,\tau)$, is a weak equivalence. It follows that the adjunction is a Quillen equivalence.

For (2), we repeat the argument above. Because $\int \Phi$ is only a semi-model category over $\int \Psi$, relative cofibrancy of $O$ and $O'$ are required for the argument that $\phi_!$ is left Quillen, and for the argument that $\eta$ and $\epsilon$ are weak equivalences (because only weak equivalences between relatively cofibrant objects are well-behaved in relatively left proper semi-model categories).

For (3), we follow the same proof. Now, since $\int \Phi$ has less structure, one needs the stronger cofibrancy assumption on $O$ and $O'$. For both (2) and (3), the argument reduces to the cube lemma for semi-model categories, Lemma \ref{lemma:cube-lemma}.
\end{proof}

We will soon see numerous applications of this result, in the settings of modules, operads, and 2-monads. To make $\int \Phi$ left proper, we rely on \cite[Theorem 2.14]{batanin-berger} which relies on the notions of $h$-monoidality of \cite[Definition 1.11]{batanin-berger}. Specifically, \cite[Theorem 8.1]{batanin-berger} and its generalization in \cite[Theorems 2.14, 4.17]{companion} can be used to make global full model structures over base model categories $\M$ satisfying the monoid axiom, which are moreover left proper when $\M$ is strongly $h$-monoidal, and relatively left proper when $\M$ is $h$-monoidal. The proof above relied on the following two lemmas. The first is well-known, and the second is a straightforward extension of \cite[Lemma 5.2.6]{hovey-book} to the context of semi-model categories.

\begin{lemma} \label{lemma:pushout in Gr}
Let the following be a pushout diagram in $\int \Phi$
\[
\xymatrix{
(A,B) \ar[r]^{(\psi,f)} \ar[d]_{(\phi,g)} \po & (A',B') \ar[d]^{\gamma} \\
(A'',B'') \ar[r]_{\alpha} & (C,D)
}
\]
Then in $\Phi(C)$, $D$ can be realized as the push forward of the pushout of the following span in $\Phi(A)$: $\phi^*(B'') \gets B \to \psi^* B'$. Furthermore, $D$ is also the pushout of the span $\alpha_!(B'') \gets \alpha_! \phi_!(B) = \gamma_! \psi_!(B) \to \gamma_!(B')$.
\end{lemma} 

\begin{proof}
The proof boils down to the universal property of the pushout, and uses the functors displayed in the commutative diagram above (and the adjunctions they induce, like $(\alpha_!,\alpha^*)$) to shift between $\Phi(A)$ and $\Phi(C)$.
\end{proof}

\begin{lemma} \label{lemma:cube-lemma}
Suppose $\calD$ is a semi-model category (resp. semi-model category over $\M$) and all objects in the following morphism $h$ of pushout squares are cofibrant (resp. cofibrant in $\M$).
\[
\xymatrix{A_1 \ar[rr]^{f_1} \ar[dd] \ar[dr]^{h_1} && B_1 \ar[dd] \ar[dr]^{h_2} & \\
& A_2 \ar[rr]^(0.3){f_2} \ar[dd] && B_2 \ar[dd] \\
C_1 \ar[rr] \ar[dr]^{h_3} && P_1 \ar[dr] & \\
& C_2 \ar[rr] && P_2}
\]
If each morphism $h_i$ is a weak equivalence and either $f_1$ or $f_2$ is a cofibration, then the induced morphism of pushouts $P_1\to P_2$ is a weak equivalence.
\end{lemma}

\begin{proof}
This result is claimed in \cite{spitzweck-thesis} (page 10) for Spitzweck's version of semi-model categories, but because our version is more general, and because we also state the result for semi-model categories \textit{over} $\M$, we will sketch the proof. The proof proceeds exactly as in \cite[Lemma 5.2.6]{hovey-book}. One defines a Reedy category $\cat B$ of spans $c\gets a\to b$, and the Reedy semi-model structure (resp. semi-model structure over $\M$) on $\cat D^{\cat B}$ from \cite[Theorem 3.12]{barwickSemi} or \cite[Proposition 3]{spitzweck-thesis}. Then we use the definition of cofibrations (via latching objects) in the Reedy semi-model structure (resp. semi-model structure over $\M$) to conclude that $C_i\gets A_i \stackrel{f_i}{\to}B_i$ are both cofibrant diagrams above. Thus, it suffices to prove that the colimit functor $\colim: \cat D^{\cat B} \to \cat D$ is left Quillen, and to apply Ken Brown's lemma \cite[Proposition 12.1.6]{fresse-book}, which, for semi-model categories over $\M$ says that left Quillen functors preserve weak equivalences between objects that are cofibrant in $\M$. This follows by \cite[12.1.8]{fresse-book} since the right adjoint of $\colim$ (namely: the constant diagram functor) preserves fibrations and trivial fibrations, by unpacking the definition of a matching object.
\end{proof}

We conclude this section with an application to strictification results in category theory. We fix $\cat K$ to be the canonical model structure on the category of small categories \cite{lack} (1.5). In fact, any canonical model structure on a locally finitely presentable 2-category would work, e.g., a category of presheaves $[\cat C^{op},Cat]$. Lack produced a model structure on the category $Mnd_f(\cat K)$ of finitary (that is, filtered-colimit preserving) 2-monads on $\cat K$. This is, in fact, transferred from a model structure on $End_f(\cat K)$ of finitary endo-2-functors of $\cat K$, as follows. The forgetful functor $U: Mnd_f(\cat K) \to End_f(\cat K)$ has a left adjoint, given by the free monoid functor, and the Lack model structure on $Mnd_f(\cat K)$ is defined so that weak equivalences and fibrations are created by $U$ \cite{lack} (2.4).

For any $T\in Mnd_f(\cat K)$, we have $\Phi(T):=\Alg_T(\cat K)$, defining a functor $\Phi$. Lack proved that every $\Phi(T)$ has a model structure transferred from $\cat K$ \cite{lack} (5.1). One $T$ in particular describes the Grothendieck construction $\int \Phi$, and hence we have a global model structure transferred from the product model structure $\M = End_f \times \cat K$. 

The cofibrant objects of $\int \Phi$ (and of the vertical structures $\Alg_T(\cat K)$ and the horizontal model structure $Mnd_f(\cat K)$) are the `flexible' ones \cite{lack} (5.6, 5.12, 7.5). Fortunately, in $\M$, all objects are cofibrant, hence all objects in $\int \Phi$ are relatively cofibrant. Hence, we may use Theorem \ref{theoremquillenequivfibers}.
We get the following result, which appears to be new, and complements \cite[Proposition 8.4]{lack}.

\begin{corollary} \label{cor:Lack-Quillen-equivalence}
Let $f: T\to T'$ be a weak equivalence in $Mnd_f(\cat K)$. Let $T$ and $T'$ be flexible monads \cite{lack} (7.5). Then $f$ induces a Quillen equivalence between $T$-algebras and $T'$-algebras with the Lack model structures \cite{lack} (5.5).
\end{corollary}

\begin{proof}
This follows immediately from Theorem \ref{theoremquillenequivfibers}(3).
\end{proof}

Several classical strictification theorems are a consequence of this Quillen equivalence. For example, the classical strictification from lax monoidal to strict monoidal categories can be seen to be a special case of Corollary \ref{cor:Lack-Quillen-equivalence} \cite{lack} (5.8). With a bit of work, one could attempt to recover the coherence results of John Power for bicategories using this machinery. We conclude with an observation about left properness.

\begin{example} \label{example:lack-not-left-proper}
Let $\cat K$ be the canonical model structure on the 2-category of 2-categories. Lack's model structure on $Mnd_f(\cat K)$ is not relatively left proper. If it were, then one could expand Corollary \ref{cor:Lack-Quillen-equivalence} to all finitary 2-monads, since every object is underlying cofibrant. However, strictification fails for the following pair of monads. Let $T$ be the free monoid monad on $\cat K$, with respect to the Cartesian monoidal structure. Let $T'$ be the free monoid monad with respect to the Gray monoidal structure. Then, for every $X\in \cat K$, we have $T(X)\simeq T'(X)$, and consequently $T$ and $T'$ are weakly equivalent. 
However, strict 2 monoids are \textit{not} equivalent to Gray monoids, because restriction is not essentially surjective in the homotopy category. For example, a braided monoidal structure does not come from a strict structure.
\end{example}

Of course, it is possible that for \textit{some} base model categories $\cat K$, that $Mnd_f(\cat K)$ is relatively left proper. This leads to an interesting open problem.

\begin{problem}
Characterize the 2-categories $\cat K$ such that $Mnd_f(\cat K)$ is relatively left proper. 
\end{problem}

\subsection{Left cofinal Quillen functors}

\begin{definition}\label{definitionfunctorpreseringcofibrantreplacementsofterminalobject}
	Let $B$ and $C$ two model categories and $G: B \to C$ a left Quillen functor. $G$ is a left cofinal Quillen functor if $\mathbb{L}G(1)$ is contractible, where $\mathbb{L}G$ is the left derived functor of $G$ and $1$ is the terminal object in $B$.
\end{definition}

\begin{remark}
	Recall from \cite[Definition 5.6]{batanin-deleger} that a morphism of polynomial monads $f: S \to T$ is homotopically cofinal if the nerve of the classifier $T^S$ is contractible. This implies that $f_! : \Alg_S(\sSet) \to Alg_T(\sSet)$ is a left cofinal Quillen functor. This notion is crucial to get multiple delooping results \cite{deleger}. It might happen that a Grothendieck construction is equivalent to the category of algebras of a polynomial monad but not its fibers. Assume that, using polynomial monads, we get a cofinality result on the global model structure. We would like to deduce from it a cofinality result between the fibers. This is the point of this section and Theorem \ref{theoremleftcofinalityfromglobaltolocal}.
\end{remark}

\begin{lemma}\label{propositionquillenequivalencecofinal}
	The left adjoint of a Quillen equivalence is a left cofinal Quillen functor.
\end{lemma}

\begin{proof}
	Let $G: B \to C$ be the left adjoint of a Quillen equivalence and $V$ be its right adjoint. Let $b \in B$ be a cofibrant replacement of the terminal object. Since $b \to V(1)$ is a weak equivalence, $G(b) \to 1$ is a weak equivalence. Moreover, $G(b)$ is cofibrant.
\end{proof}

Let $B$ and $C$ be two model categories and $\Phi: B^{op} \to \Cat$ and $\Psi: C^{op} \to \Cat$ two functors. We want to find sufficient conditions for a functor between Grothendieck constructions
\[
F: \int \Phi \to \int \Psi \ ,
\]
to induce a functor between fibers
\[
\mathcal{F} : \Phi(b) \to \Psi(c)
\]
which is a left cofinal Quillen functor.

\begin{definition}
	Let $\mathcal{K}$ be the category whose
	\begin{itemize}
		\item objects are pairs $(B,\Phi)$ where $B$ is a category and $\Phi: B^{op} \to \mathrm{CAT}$ is a functor
		
		\item morphisms $(B,\Phi) \to (C,\Psi)$ are given by pairs $(G,H)$ where $G: B \to C$ is a functor and $H: \Phi \Rightarrow \Psi G$ is a natural transformation
		\[
		\xymatrix{
			B^{op} \ar[rr]^-{\Phi} \ar[rd]_-{G} \xtwocell[rr]{}<>{<3.2>\ H} && \mathrm{CAT} \\
			& C^{op} \ar[ru]_-{\Psi}
		}
		\]
		
		\item composites are given by
		\[
		(G_1,H_1) (G_2,H_2) = (G_1 G_2 , (H_1 \rhd G_2) H_2)
		\]
		where $\rhd$ is the right whiskering \cite{BBFW}.
	\end{itemize}
\end{definition}

\begin{definition}\label{definitiondecomposable}
	A functor between Grothendieck constructions $F : \int \Phi \to \int \Psi$ is \emph{decomposable} if it is the image of a morphism $(G,H)$ in $\mathcal{K}$ through the functor $\mathcal{K} \to \mathrm{CAT}$ which sends a pair $(B,\Phi)$ to the category $\int \Phi$. We  will write $F = G \rtimes H$.
\end{definition}

\begin{theorem}\label{theoremleftcofinalityfromglobaltolocal}
	Assume we have a functor between Grothendieck constructions
	\[
	F : \int \Phi \to \int \Psi
	\]
	
	Assume that $F$ is decomposable with $F = G \rtimes H$. Let $b$ and $c$ be cofibrant replacements of the terminal object in $B$ and $C$ respectively. Let $\gamma: G(b) \to c$ be a weak equivalence such that $\gamma^*$ reflects weak equivalences and the unique morphism $0 \to \gamma^*(0)$ is a weak equivalence.
	
	If $F$ is a left cofinal Quillen functor, then the composite functor
	\[
	\Phi(b) \xrightarrow{H_b} \Psi G(b) \xrightarrow{\gamma_!} \Psi(c)
	\]
	is also a left cofinal Quillen functor.
\end{theorem}

\begin{proof}
We will prove that both $H_b$ and $\gamma_!$ are left cofinal Quillen functors.
	
If $x$ is a cofibrant replacement of the terminal object in $\Phi(b)$, then, since $b$ is a cofibrant replacement of the terminal object in $B$, $(b,x)$ is a cofibrant replacement replacement of the terminal object in $\int \Phi$. Since $F$ is a left cofinal Quillen functor, $F(b,x)$ is a cofibrant replacement of the terminal object. Therefore, $H_b(x)$ is a cofibrant replacement of the terminal object in $\Psi G (b)$.

Since, by Theorem \ref{theoremquillenequivfibers}, the adjunction $\gamma_! \dashv \gamma^*$ is a Quillen equivalence, $\gamma_!$ is a left cofinal Quillen functor thanks to Lemma \ref{propositionquillenequivalencecofinal}. 
	
It is obvious that the composite of two left cofinal Quillen functors is a left cofinal Quillen functor. This concludes the proof.
\end{proof}

\subsection{Consequences} \label{sec:consequences}

Thanks to Theorem \ref{thm:global-model-examples} we have an immediate corollary of Theorems \ref{thm:global-to-local} and \ref{thm:properness}.

\begin{corollary}
Assume $\M$ is a compactly generated monoidal model category satisfying the monoid axiom. 

\begin{enumerate}
\item Then the following horizontal model structures exist:
\begin{enumerate}
\item The category of $R$-algebras for a commutative monoid $R$ in $\M$.
\item The category of non-symmetric operads in $\M$. 
\item The category of pairs $(O,P)$ where $O$ and $P$ are non-symmetric operads.
\item The category of constant-free symmetric operads.
\item The category of constant-free $n$-operads.
\item The category of commutative monoids, if $\M$ satisfies the commutative monoid axiom.
\end{enumerate}

\item Furthermore, the following vertical model structures exist, transferred in the obvious way.

\begin{enumerate}
\item The categories of $A$-modules (left, right, or bimodules) over an $R$-algebra $A$, where $R$ is a commutative monoid in $\M$.
\item The category of $O$-algebras or left $O$-modules, where $O$ is a non-symmetric operad.
\item The category of $O-P$-bimodules, where $O$ and $P$ are non-symmetric operads.
\item The category of constant-free left $O$-modules, where $O$ is a constant-free symmetric operad.
\item The category of constant-free left $O$-modules, where $O$ is a constant-free $n$-operad.
\end{enumerate}

\end{enumerate}

If in addition $\M$ is $h$-monoidal then these model structures are relatively left proper. If in addition $\M$ is strongly $h$-monoidal then they are left proper. Furthermore, Theorem \ref{theoremquillenequivfibers} provides a rectification result, whereby weak equivalences in the horizontal model structures yield Quillen equivalences between vertical model structures, such as categories of algebras over weakly equivalent operads.
\end{corollary}

Of course, many of these results were known previously, but we wanted to list them all to illustrate the power of our main results to apply to so many settings. Furthermore, the results regarding constant-free operads do not seem to have appeared elsewhere. 

\subsection{From semi to full}

Theorem \ref{thm:global-to-local} shows that if $\int \Phi$ is a semi-model category and ${O}$ is cofibrant in $\int \Phi$, then there is a vertical semi-model structure on $\Phi({O})$. In this section, we show how to get full model structures on the fibers $\Phi(O)$. We will use Lemma \ref{lemma:factorization-suffices}.

We will assume we are in the nice situation of Figure (\ref{figure:adjoints}). In particular, we assume that the inclusion functor $i:\cat{B} \to \int \Phi$ has a left adjoint $E$, with the property that the initial object in $\Phi(E(O,A))$ is $A$. Recall that in the setting of operads, $E(O,A)$ is the enveloping operad $O_A$ discussed in \cite{white-yau}. For the setting of monoids and modules, $E(R,M)$ is the enveloping algebra. 
Note that, by Proposition \ref{prop:i-right-Quillen}, the existence of a left adjoint $E$ implies $i$ is a right Quillen functor.

Note that $\Phi(O)$ cannot always have a full model structure. For example, if $O$ is the operad for non-reduced symmetric operads, then $O$-alg has a semi-model structure, but not a full model structure, when the base model category is $Ch(\bbF_p)$ \cite[Example 2.9]{batanin-white-eilenberg-moore}. Thus, our result must have additional conditions.

\begin{lemma} \label{lemma:OA-factorization}
Assume that $\int \Phi$ is homotopically structured and admits a cofibrantly generated global semi-model structure. Assume that, for every morphism $\phi$ in $\cat{B}$, $\phi^*$ preserves fibrations and weak equivalences.
Then, for any morphism $f:A\to B$ in $\Phi(O)$, if the induced morphism $E(O,f):E(O,A)\to E(O,B)$ admits a (trivial cofibration, fibration) factorization in $\cat{B}$, then so does $f$ in $\Phi(O)$.
\end{lemma}

Below, let $i_P$ denote the initial object in $\Phi(P)$.

\begin{proof}
Denote the hypothesized factorization as $\beta \circ \alpha:E(O,A)\to P\to E(O,B)$. The morphism $\alpha: E(O,A) \to P$ yields a morphism $i(\alpha) = (\alpha,g):(E(O,A),A) \to (P,i_P)$. Since $P \cong E(P,i_P)$, the morphism $\beta:P \to E(O,B)$ yields a morphism $i(\beta)=(\beta,h):(P,i_P) \to (E(O,B),B)$, where $h:i_P \to \beta^*(B)$ is a morphism in $\Phi(P)$.
Our proof relies on the observation that $\Phi(E(O,A)) \cong A \downarrow \Phi(O)$, and similarly for $B$.

Observe that $(\alpha,g)$ factors as $(E(O,A),A) \to (E(O,A),\alpha^*(i_P)) \to (P,i_P)$, where the first morphism is $(id,g)$ and the second is $(\alpha,id)$. We claim that $f$ factors as $A\to \alpha^*(i_P) \to B$, where the first morphism is $g$, the second is $\alpha^*h$, and we abuse notation to allow $\alpha^*(i_P)$ to refer to both an object in $A \downarrow \Phi(O)$ and the corresponding object in $\Phi(O)$. 
We must prove that $g$ is a trivial cofibration, that $\alpha^*h$ has the correct codomain, and that $\alpha^*h$ is a fibration. 

To understand the morphism $\alpha^*h$, view $B$ in $B \downarrow \Phi(O)$ as $id_B$, so that restriction to $\Phi(E(O,A))$ works by composing with the morphism $f:A\to B$. Hence, $\alpha^*(\beta^*(B))$ is isomorphic to $B$, so $\alpha^*h:\alpha^*(i_P)\to \alpha^*(\beta^*(B))\cong B$ has the correct codomain. Since $\beta$ is a horizontal fibration, this implies $i(\beta)=(\beta,h)$ is a global fibration, by Proposition \ref{prop:i-right-Quillen}. Hence, $h$ is a vertical fibration. Next, $\alpha^*(h)$ is a fibration because $\alpha^*$ preserves fibrations.

We turn to the morphism $g$. We know that $i(\alpha) = (\alpha,g)$ is a trivial cofibration, since $i$ is left Quillen by Proposition \ref{prop:stanculescu}(ii). Hence, $g$ is a trivial cofibration, thanks to Lemma \ref{lemma:cof-of-pairs}. This completes our factorization of $f$.
\end{proof}

\begin{theorem} \label{thm:semi-to-full}
Assume the functor $E:\int \Phi \to \cat{B}$ exists and $\int \Phi$ has a cofibrantly generated global semi-model structure such that, for every $\phi$ in $\cat B$, the functors $\phi^*$ preserve weak equivalences and fibrations. If $\cat{B}$ has a full model structure then for any $O \in \cat{B}$, there is a full model structure on $\Phi(O)$.
\end{theorem}

\begin{proof}
Theorem \ref{thm:global-to-local} provides a semi-model structure on $\Phi(O)$. Lemma \ref{lemma:factorization-suffices} proves that, if any morphism $f:A\to B$ in $\Phi(O)$ has a factorization into a trivial cofibration followed by a fibration, then this semi-model structure is in fact a full model structure. Since every morphism $E(O,A)\to E(O,B)$ has such a factorization (since $\cat{B}$ has a full model structure), Lemma \ref{lemma:OA-factorization} proves that $f$ has a factorization, so $\Phi(O)$ has a full model structure.
\end{proof}

\begin{remark}
This corollary does not contradict our usual counterexamples. For example, if $\M = Ch(\mathbb{F}_2)$ then there is NOT a full model structure on $\cat B$, the category of symmetric operads. There is a full model structure on constant free symmetric operads \cite[Section 9]{batanin-berger} but such operads cannot encode strict commutative monoids, so no contradiction arises.
\end{remark}

\begin{remark}
The conditions of Theorem \ref{thm:semi-to-full} are easy to check. The global semi-model structure is always transferred in our settings, so is automatically cofibrantly generated. The conditions on functors $\phi^*$ are always satisfied for us because weak equivalences and fibrations in vertical structures $\Phi(O)$ (e.g., $\algom$) are created in a common underlying model category $\M$. And many previous authors have found ways to make categories of operads into full model categories. This is much easier than making $\algom$ into a model category for an arbitrary $O$, since the operad whose algebras are symmetric operads (or non-symmetric, cyclic, modular, $n$-operads, etc.) is $\Sigma$-cofibrant. Thus, Theorem \ref{thm:semi-to-full} reduces the difficulty of producing model structures on categories of operad-algebras to the problem of transferring a model structure to the category of operads.
\end{remark}

\begin{remark}
When the conditions of Theorem \ref{thm:semi-to-full} are satisfied, the conclusion of the theorem is that we have full model structures on the base and fibers, but still only a semi-model structure on $\int \Phi$. In this position, one can apply the results mentioned in the introduction \cite{stanculescu, harpaz-prasma-integrated, cagne-mellies} to produce the total model structure on $\int \Phi$, if certain conditions on the adjunctions $\phi_! \dashv \phi^*$ are met. When the total model structure exists, it coincides with the global structure, promoting the latter from a semi-model structure to a full one.
\end{remark}

\begin{remark}
Our careful treatment of the global (semi-)model structure above was motivated by two common situations. First, it may happen that our ambient model category $\M$ is only well-behaved enough for $Gr(T)$ to have a semi-model structure, even when the polynomial monad $T$ is very nicely behaved. For example, if $T$ is the monad for monoids, but $\M$ fails to satisfy the monoid axiom, then $\Alg_T$ and $\int \Phi$ have only semi-model structures. The second common situation is when we want to study symmetric operads and their algebras. We consider this situation in the next section.
\end{remark}

\section{Grothendieck constructions for symmetric operads} \label{sec:non-tame}

In this section, we study the Grothendieck construction $\int \Phi$ in cases where the global structure is provably only a semi-model structure and not a full model structure. Such situations are common when $\M$ is a general model category (not an extremely nice one like chain complexes over a field of characteristic zero, simplicial sets, or others detailed in \cite{white-yau} where all operads are admissible). We first focus on the category of pairs $(O,A)$ where $O$ is a symmetric operad and $A$ is an $O$-algebra, and later focus on modules and bimodules. We distinguish between classical, reduced (meaning $O(0) = \ast$ is the terminal object of $\M$), and constant-free (meaning $O(0)=\emptyset$ is the initial object of $\M$) symmetric operads. We unify these seemingly disparate settings into a single framework, reprove known results in these settings (sometimes with lesser hypotheses), and prove previously unknown results.

In these situations, $\int \Phi$ provably has only a semi-model structure, not a full model structure. However, $\cat{B}$ is a semi-model category \textit{over} the category of symmetric collections (better than being just a semi-model category), and similarly, in this situation the Grothendieck construction is better behaved than one might initially guess. We fix an ambient model category $\M$.

For the case of pairs $(O,A)$ where $O$ is a symmetric operad valued in some monoidal model category $\M$, and $A$ is an $O$-algebra, the global semi-model structure is transferred from the product $\M^\Sigma \times \M$, where $\M^\Sigma = \Pi_{n\geq 0} \M^{\Sigma_n}$ is the category of symmetric collections in $\M$, in which each $\M^{\Sigma}$ is given the projective model structure. For the case of pairs $(O,M)$ where $M$ is a left $O$-module, the global semi-model structure is transferred from $\M^\Sigma \times \M^\Sigma$. With different restrictions on arity zero, we obtain three flavors of symmetric operads: constant-free, reduced, and classic (or non-reduced) \cite{batanin-berger}.

For the case of symmetric $\fC$-colored operads, for a fixed color set $\fC$, valued in a monoidal model category $\M$, the transfer for algebras is from $\M^{\Sigma_\fC} \times \M^{\fC}$, and modules similarly \cite{white-yau}. The color set is irrelevant for our results, so we omit it from the notation. Let $SOp$ (resp. $SOp_0$, resp. $SOp_{CF}$) denote the polynomial monad whose algebras are ($\fC$-colored) symmetric operads (resp. reduced symmetric operads, resp. constant-free symmetric operads). The polynomial monads $SOp_0$ and $SOp_{CF}$ are tame, though $SOp$ is not \cite{batanin-berger}. Consider the corresponding polynomials for the Grothendieck construction (Proposition \ref{prop:poly-for-Gr}) for pairs $(O,A)$ where $A$ is an $O$-algebra. Consider also the polynomials for the Grothendieck construction of pairs $(O,M)$ where $M$ is a left $O$-module.

In \cite[Proposition 4.25]{companion}, we prove that $Gr(SOp_{CF},mod)$ is quasi-tame, and hence there is a global model structure in this case. However, $Gr(SOp_0,alg)$, $Gr(SOp,alg)$, $Gr(SOp_0,mod)$, $Gr(SOp,mod)$ are not quasi-tame, and there is not a transferred model structure in these cases, as we now show.

\begin{example}
Consider the monad $\cat P = Gr(SOp(\fC))$ for the Grothendieck construction of $\fC$-colored symmetric operads and their algebras, for non-empty $\fC$. The global structure on $\int \Phi$ yields a semi-model category that is not a model category. If it were, then taking $\M = Ch(\mathbb{F}_2)$, we would obtain a full model structure on $\int \Phi$, which would imply a full model structure on $\Phi(O)$ for any $O$. In particular, there would be a vertical model structure on commutative differential graded algebras, contradicting \cite[Section 5.1]{white-commutative-monoids}. Similar counterexamples exist for the reduced case and the case of modules (instead of algebras). The monad $Gr(SOp_{CF},alg)$ appears not to be quasi-tame, but we do not have a counterexample in this setting.
\end{example}

Fortunately, we still have global semi-model structures in these cases. Furthermore, the base category of symmetric operads (whether reduced or not) are semi-model \textit{over} appropriate categories of symmetric collections (Definition \ref{defn:semi-model-over}). 
We now summarize our results from Section \ref{sec:global-model} to these settings. We slightly abuse notation and write `the category of symmetric operads' to mean any of the three variants discussed above (and similarly for symmetric collections).

\begin{proposition} \label{prop:symmetric-operad-results}
Fix a cofibrantly generated monoidal model category $\M$. Let $\int \Phi$ denote any of the five examples above that fail to have global model structures.
\begin{enumerate}
\item The category of symmetric operads has a transferred semi-model structure over the category of symmetric collections.
\item For every $\Sigmac$-cofibrant colored operad $O$, the category of $O$-algebras has a transferred semi-model structure. Similarly, the category of left $O$-modules has a  transferred semi-model structure.
\item If $f: O\to O'$ is a weak equivalence of $\Sigmac$-cofibrant operads then $f$ induces a Quillen equivalence between $O$-algebras and $O'$-algebras (resp. modules).
\end{enumerate}
\end{proposition}

\begin{proof}
(1) is a special case of \cite[Theorem 2.2.2]{Reedy-paper}. 
(2) and (3) follow from (1) exactly like Proposition \ref{prop:global-to-vert-over} and \ref{theoremquillenequivfibers}.
\end{proof}

Proposition \ref{prop:symmetric-operad-results} reproves known results using the global context of the Grothendieck construction. Fresse works in the context of semi-model categories, rather than semi-model categories over a base, so (1)-(2) may be viewed as improvements of results in \cite{fresse-book} (Chapter 12), though surely these results were known to both Fresse and Spitzweck. In (3), we reprove \cite[Theorem 12.5.A]{fresse-book}, though without the assumption that the unit of $\M$ is cofibrant. Like \cite[Theorem 16.2.A]{fresse-book}, we also have a result for bimodules.

\begin{corollary}
Let $P$ and $R$ be two $\Sigma$-cofibrant symmetric operads (colored by $\fC$ and $\fD$ respectively).
\item If $f: R\to S$ is a weak equivalence of $\Sigma$-cofibrant colored operads, then it induces a Quillen equivalence between $(P,R)$-bimodules and $(P,S)$-bimodules.
\item If $g:P\to Q$ is a weak equivalence of $\Sigma$-cofibrant colored operads, then it induces a Quillen equivalence between $(P,R)$-bimodules and $(Q,R)$-bimodules.
\end{corollary}

\begin{proof}
This follows from the same results that Proposition \ref{prop:symmetric-operad-results} does.
\end{proof}

Most of the previous work on the homotopy theory of symmetric operads and their algebras/modules deals with reduced or classical operads. The case of constant-free symmetric operads seems much less studied, and we are unaware of any paper that proves the results above in the constant-free context. It appears our results about semi-model structures on algebras and modules over constant-free symmetric operads are new.

One could analogously study \textit{connected} $(P,R)$-bimodules via the Grothendieck construction, or algebras/modules over a dioperad or permutad. Additional restrictions are required in these latter settings. We leave the details to the reader.

We conclude this section by reiterating that known counterexamples prevent us from having a model structure on the category of symmetric operads or reduced symmetric operads in general model categories $\M$. Hence, global model structures are not possible in general. However, global semi-model structures \textit{over} a suitable base may be possible. To produce such structures would require the combination of the theory of unary tameness \cite{Reedy-paper} with the theory of quasi-tameness. We leave the development of unary quasi-tameness to the interested reader.

\begin{problem}
Develop a theory of unary quasi-tame substitudes, and prove that the category of algebras over such a substitude admit a transferred semi-model structure over the base model category.
\end{problem}

\section{Application to Twisted Modular Operads} \label{sec:twisted}

We conclude the paper with an application of the machinery of the preceding sections to put a model structure on the category of twisted modular operads. We recall the relevant terminology from \cite{getzler-kapranov}. 
Like Getzler and Kapranov, we fix a field $k$ of characteristic zero \cite{getzler-kapranov} (1.1), and we work in the model category of $\mathbb{Z}$-graded chain complexes of $k$-vector spaces. Getzler and Kapranov define categories of cyclic operads and modular operads \cite{getzler-kapranov} (1.5, 2.20), and these categories are known to admit transferred model structures \cite{batanin-berger} (9.7, 10.1). Getzler and Kapranov define the notion of a \textit{hyperoperad} \cite{getzler-kapranov} (4.1) to encode more general algebraic structure than one can encode with an operad. We note that the Baez-Dolan plus construction produces hyperoperads \cite{batanin-berger}. 

We now follow \cite[Section 4]{getzler-kapranov}. Getzler and Kapranov define a specific hyperoperad $\mathfrak{D}$ to be the monoidal unit of $\M$ in every degree. % 4.1.2
Its algebras are modular operads. % 4.1.3
They fix a cocycle $D$ and define an object $\tilde{D}$ (which we will denote $\cat O$) such that $\cat O$-algebras are twisted modular operads.

There is no known Set-valued operad whose algebras are twisted modular operads, and so the usual methods of transferring a model structure to the category of twisted modular operads cannot be used. Nevertheless, the techniques of this section can be used to place a model structure on this category, via the Grothendieck construction. As far as the authors know, this model structure has not appeared elsewhere in the literature.

\begin{theorem} \label{thm:twisted-modular}
The category of twisted modular operads of \cite{getzler-kapranov} has a vertical model structure.
\end{theorem}

\begin{proof}
As before, let $\M$ be the category of chain complexes over a field of characteristic zero. Let $T$ be the polynomial monad for non-symmetric operads, and note that $T^+$ is the monad for hyperoperads \cite[Definition 5.21]{florian2023}. 
The Grothendieck construction $\int \Phi$ in this case gives pairs $(\cat{H},O)$ where $\cat{H}$ is a hyperoperad and $O$ is an algebra over $\cat{H}$. 
Since $\int \Phi$ can be encoded as a category of algebras over a colored operad, and all colored operads are admissible in $\M$ \footnote{Technically, here we need a product of copies of $\M$, but it is easy to see that all operads are admissible because we are in characteristic zero and all groups act freely.} \cite[Section 8]{white-yau}, there is a full model structure on $\int \Phi$. Theorem \ref{thm:global-to-local} then implies there is a full model structure on all fibers, including on the category of twisted modular operads.
\end{proof}

\end{document}